\documentclass[12pt]{imsart}
\usepackage{amsthm,amsmath}
\usepackage[colorlinks,citecolor=blue,urlcolor=blue]{hyperref}
\usepackage{datetime2}
\usepackage{amsmath,amsfonts,amssymb,amsthm}

\usepackage{enumitem}

\RequirePackage{amsthm,amsmath,amsfonts,amssymb}
\RequirePackage[numbers]{natbib}
\RequirePackage[colorlinks,citecolor=blue,urlcolor=blue]{hyperref}
\RequirePackage{graphicx}
\usepackage{lineno,hyperref}

\startlocaldefs
\modulolinenumbers[5]

\usepackage[utf8]{inputenc}
\usepackage[english]{babel}

\usepackage[margin=1in]{geometry}
\usepackage{bm}
\usepackage{graphicx}
\usepackage{amssymb,amsmath,amsthm,mathrsfs}
\usepackage{epstopdf}
\usepackage{algorithm}
\usepackage{amsfonts,delarray}
\usepackage{algorithm,algcompatible,amsmath}
\usepackage{rotating,color}
\usepackage{subfigure}
\usepackage{multirow}
\usepackage{titlesec,textcase}
\usepackage{longtable}
\usepackage{bbm}
\usepackage{rotating}

\newtheorem{Theorem}{Theorem}[section]
\newtheorem{Lemma}[Theorem]{Lemma}

\newtheorem{Proposition}[Theorem]{Proposition}

\newtheorem{Definition}[Theorem]{Definition}

\theoremstyle{definition}

\RequirePackage[normalem]{ulem} 

\definecolor{rp}{RGB}{83,54,106}

\def\boxit#1{\vbox{\hrule\hbox{\vrule\kern6pt\vbox{\kern6pt#1\kern6pt}\kern6pt\vrule}\hrule}}

\endlocaldefs

\allowdisplaybreaks

\begin{document}
\begin{frontmatter}
\title{Central limit theorem for the global clustering
coefficient of random geometric graphs}

\runtitle{CLT for clustering coefficient of RGGs}
\runauthor{ M. Yuan }

\begin{aug}

\author[B]{\fnms{Mingao} \snm{Yuan}\ead[label=e2]{myuan2@utep.edu
} },\ \ \ 
\author[C]{\fnms{Md. Niamul Islam} \snm{Sium}\ead[label=e3]{msium@miners.utep.edu
}}


\address[B]{Department of Mathematical Sciences,
The University of Texas at El Paso, El Paso, Texas 79968, USA
\printead{e2}}

\address[C]{Department of Mathematical Sciences,
The University of Texas at El Paso, El Paso, Texas 79968, USA
\printead{e3}}

\end{aug}
 
\begin{abstract}
The global clustering coefficient serves as a powerful metric for the structural analysis and comparison of complex networks. Random geometric graphs offer a  realistic  framework for representing the spatial constraints and geometry often found in real-world network datasets. In this paper, we establish a central limit theorem for the global clustering coefficient of random geometric graphs. Our main result identifies the centering and scaling sequences required for convergence in law to the standard normal distribution. Our approach varies by regime: in the dense case, we employ the Lyapunov CLT; in the intermediate case, we utilize the asymptotic theory of 
$U$-statistics with sample-size-dependent kernels; and in the sparse regime, we use the method of moments to derive the asymptotic distribution. Notably, the convergence rates for non-uniform and uniform random geometric graphs diverge in the dense regime, yet they coincide in the sparse regime. In addition,  we find that the global clustering coefficient for both uniform and non-uniform RGGs is asymptotically equal to $3/4$.
\end{abstract}



\end{frontmatter}

\section{Introduction}

Random geometric graphs have emerged as powerful models for networks with spatial or geometric structure. In these models, nodes are typically distributed according to a probability distribution over a metric space, and an edge connects two nodes if the distance between them is below a certain threshold \cite{Penrose2003,Gilbert1961,BB24}. When the node distribution is uniform, the model is referred to as a uniform random geometric graph. Conversely, if the distribution is non-uniform, it is termed a non-uniform random geometric graph.  Unlike classical random graph models such as the Erd\H{o}s-R\'enyi model where edges form independently with a fixed probability \cite{ErdosRenyi1959}, random geometric graphs naturally incorporate spatial constraints and geometric dependencies that arise in many real-world networks \cite{Penrose2003,Barthelemy2011,ZLKG14}. The spatial embedding induces correlations among edges that share common endpoints, leading to structural properties fundamentally different from classical random graphs \cite{DallChristensen2002}. These models find applications in diverse domains including wireless sensor networks \cite{Gupta2000}, biological systems \cite{H08}, public health \cite{BM08, Zheng2025}.

The global clustering coefficient is a foundational measure for quantifying network topology, reflecting the tendency of nodes to form tightly connected clusters \cite{WS98,N03}. Defined as the ratio of three times the number of triangles to the number of 2-paths [5], this metric has become central to network data analysis  \cite{RPW05,N03,LCYC14,SBL13}. For example, \cite{RPW05} used clustering coefficients to differentiate between network structures through extensive simulation studies.

Understanding the properties of the global clustering coefficient within random graph models is essential \cite{LCYC14,DallChristensen2002,Y25c}. One primary research focus involves deriving analytical expressions for this metric. For example, \cite{LCYC14} analytically derived the clustering coefficient for an online social network model, while \cite{DallChristensen2002} established an expression for a random geometric graph model constructed differently from the one presented here. Another significant research direction explores the limiting distribution of the global clustering coefficient \cite{Y25c}. To our knowledge, the clustering coefficient of non-uniform random geometric graphs—where nodes are distributed non-uniformly—remains less well understood.

In this paper, we establish a central limit theorem for the global clustering coefficient of nonuniform random geometric graphs.  We begin by characterizing the leading-order terms for the expected counts of small subgraphs, such as edges, 2-paths, and triangles. These results are foundational to our analysis and of independent interest, as they elucidate how nonuniform vertex distributions influence local subgraph architecture. Building on these characterizations, we derive the exact asymptotic limit and limiting distribution of the clustering coefficient.  Our analysis demonstrates that the global clustering coefficient in the nonuniform random geometric graph converges in probability to $3/4$
, characterizing a limit that, to our knowledge, has not been previously reported. This limit coincides with the uniform case.  We prove that, under appropriate scaling and centering, the coefficient converges in law to the standard normal distribution. Our methodology spans three distinct regimes: the Lyapunov central limit theorem for the dense case, the asymptotic theory of 
$U$-statistics with sample-size-dependent kernels for the intermediate case, and the method of moments for the sparse case.

Despite sharing identical limits and asymptotic distributions, the rates of convergence for nonuniform and uniform random geometric graphs diverge. Specifically, the nonuniform model exhibits a slower convergence rate than its uniform counterpart in the dense regime. However, in the sparse regime, these convergence rates coincide. These findings highlight a critical disparity in the underlying architecture of these models, offering deeper insight into how nonuniformity shapes the global properties of geometric graphs.

The paper is organized as follows. Our main results are presented in Section \ref{mainresultf}, while the proofs are deferred to Section \ref{secpf}.

\medskip

Notations: Throughout this paper, we adopt the Bachmann–Landau notation for asymptotic analysis. Let $c_1,c_2$ be two positive constants. For two positive sequence $a_n$, $b_n$, denote $a_n=\Theta( b_n)$ if $c_1\leq \frac{a_n}{b_n}\leq c_2$; denote $a_n=O(b_n)$ if $\frac{a_n}{b_n}\leq c_2$; $a_n=o(b_n)$ or $b_n=\omega(a_n)$ if $\lim_{n\rightarrow\infty}\frac{a_n}{b_n}=0$. Let $X_n$ be a sequence of random variables. $X_n=O_P(a_n)$ means $\frac{X_n}{a_n}$ is bounded in probability.  $X_n=o_P(a_n)$ means $\frac{X_n}{a_n}$ converges to zero in probability. $X_n\Rightarrow N(0,1)$ represents $X_n$ converges in distribution to the standard normal distribution. The notation $\sum_{i\neq j\neq k\neq l}$ represents summation over indices $i,j,k,l\in\{1,2,3,\dots,n\}$ with $i\neq j,i\neq k, i\neq l, j\neq k, j\neq l, k\neq l$. $I[E]$ is the indicator function of event $E$. $E^c$ represents the complement of event $E$. For a set $B$, $|B|$ denote the number of elements in the set $B$. For a function $g(x)$, $g^{(k)}(x)$ denotes the $k$-th derivative of $g(x)$.  We also use $f'(x)$, $f''(x)$ and $f'''(x)$ denote the first, second, and third derivatives of \(f(x)\), respectively.

\section{Main results}\label{mainresultf}

Let $\mathcal{G}=(\mathcal{V},\mathcal{E})$  be an undirected graph on $n$ vertices, where \(\mathcal{V}=\{1,2,\dots ,n\}\) is the vertex set and  \(\mathcal{E}\subseteq \{\{i,j\}:i,j\in \mathcal{V},i\ne j\}\) is the edge set.   An adjacency matrix $A$ is a square $n\times n$
 matrix used to represent a graph, where \(A_{ij}=1\) if \(\{i,j\}\in \mathcal{E}\), and \(A_{ij}=0\) otherwise. The number of edges  incident to a vertex $i$ is called its degree, denoted by $d_i$. A triangle is a set of three mutually adjacent vertices, while a 2-path is a set of three distinct vertices connected by exactly two edges.

The global clustering coefficient $\mathcal{C}_n$  is defined as the ratio of the number of triangles to the number of 2-paths \cite{WS98,N09}:
\[\mathcal{C}_n=\frac{\sum_{i\neq j\neq k}A_{ij}A_{jk}A_{ki}}{\sum_{i\neq j\neq k}A_{ij}A_{jk}}.\]
It quantifies the transitivity of the graph, representing the tendency of two neighbors of a common vertex to be connected.

A graph is considered random if its edge set is determined according to a specific stochastic process. In the classical Erd\H{o}s-R\'{e}nyi  model, every edge exists independently with a fixed, uniform probability. In contrast, random geometric graphs (RGGs) incorporate spatial constraints: vertices are distributed within a metric space, and an edge exists between two vertices if and only if their distance is less than a prescribed connectivity radius 
 \cite{DC23,GMPS23}. In this work, we focus on the random geometric  graphs defined below  \cite{DC23,GMPS23}.

\begin{Definition}\label{def1}
Let $r_n\in[0,0.5]$ be a real number and $f(x)$ be a probability density function on $[0,1]$. Given independent random variables $X_1,X_2,\dots,X_n$ distributed according to $f(x)$, the Random Geometric Graph (RGG) $\mathcal{G}_{n}(f,r_n)$ is defined as
\[A_{ij}=I[d(X_i,X_j)\leq r_n],\]
 where $A_{ii}=0$ and  $d(X_i,X_j)=\min\{|X_{i}-X_{j}|,1-|X_{i}-X_{j}|\}$.
\end{Definition}

The model $\mathcal{G}_{n}(f,r_n)$ is a one-dimensional random geometric graph, where the vertices are
 independently sampled from the density function 
$f(x)$ on the unit interval $[0,1]$. An edge exists between any two vertices within distance $r_n$ \cite{HM09,HM12,BC23}. This framework effectively captures the spatial geometry and structural dependencies characteristic of real-world networks. When  $f(x)$ is the uniform density, 
$\mathcal{G}_{n}(f,r_n)$
 corresponds to the \textit{uniform random geometric graph}. Conversely, the model is termed a \textit{non-uniform random geometric graph} when $f(x)$
 is non-constant. To date, research has focused extensively on the uniform case, largely due to its analytical tractability \cite{HRP08, GRK05, GMPS23, Z15, YF25, Y25c, Y23b}. However, non-uniform RGGs have recently gained prominence for their superior capacity to model the inhomogeneities inherent in empirical networks \cite{PBGKL23, G21, Y25}.

We now present a pivotal proposition that is fundamental to the study of the asymptotic properties of the global clustering coefficient for non-uniform random geometric graphs.

\begin{Proposition}\label{propmain}
Let $A$ be sampled from $\mathcal{G}_{n}(f,r_n)$. Suppose $f(x)=g(x)I[0\leq x\leq 1]$, where $g(x)$ is a  periodic function with period one that is bounded away from zero and has a bounded fourth derivative. In addition, we assume  
 $r_n=o(1)$. Let $\mu_n=\frac{\mathbb{E}[A_{12}A_{13}A_{23}]}{\mathbb{E}[A_{12}A_{13}]}$.  Then we have
\begin{align}\label{prope0}
    \mathbb{E}[A_{12}|X_1]&=2r_nf(X_1) + \frac{f''(X_1)}{3} r_n^3 + O(r_n^5),\\     \label{prope1}
   \mathbb{E}[A_{12}A_{13}|X_1]& = 4r_n^2f^2(X_1) + \frac{4r_n^4 }{3} f(X_1) f''(X_1) + O(r_n^6),\\
   \label{prope01}
   \mathbb{E}[A_{12}A_{23}|X_1]&=4r_n^2f^2(X_1) +\frac{r_n^4}{3}\big[4(f^{\prime}(X_1))^2+6f(X_1)f^{\prime\prime}(X_1)\big]+ O(r_n^6),\\
\label{prope3}
    \mathbb{E}[A_{12}A_{13}A_{23}|X_1]&=3r_n^2f^2(X_1)+ \frac{5 r_n^4}{12} \left[(f'(X_1))^2 + 2f(X_1) f''(X_1)\right] + O(r_n^5),\\
\label{prope4}
\mathbb{E}[A_{12}A_{13}]&= 4 r_n^2\mathbb{E}[f^2(X_1)]  + \frac{4r_n^4 }{3} \mathbb{E}[f(X_1)f''(X_1)] + O(r_n^5),\\
\label{prope2}
    \mathbb{E}[A_{12}A_{13}A_{23}]&=3r_n^2\mathbb{E}[f^2(X_1)]+ \frac{5r_n^4}{12} \mathbb{E}[(f'(X_1))^2+2f(X_1)f''(X_1)] + O(r_n^6),\\ \label{defafbf1}
    \mu_n&=\frac{3+r_n^2a_f}{4+r_n^2b_f}+O(r_n^3),
\end{align}
Here, the remainder terms $O(r_n^3)$, $O(r_n^5)$ and $O(r_n^6)$ do not depend on $X_1$,  and 
\begin{align}\label{defafbf}
a_f=\frac{5 \mathbb{E}[(f'(X_1))^2+2f(X_1)f''(X_1)]}{12\mathbb{E}[f^2(X_1)]},\hskip 1cm b_f=\frac{4 \mathbb{E}[f(X_1)f''(X_1)]}{3\mathbb{E}[f^2(X_1)]}.
\end{align}
\end{Proposition}
Proposition \ref{propmain} plays a crucial  role in our analysis  and is of independent interest for understanding the local subgraph structures of non-uniform random geometric graphs. Specifically, it characterizes the leading-order terms for the expected counts of small subgraphs—including edges, 2-paths, and triangles—thereby illuminating how non-uniform vertex distributions affect these counts. These results are also essential for deriving both the asymptotic limit and the limiting distribution of the clustering coefficient. For instance, we demonstrate in the sequel that the clustering coefficient is asymptotically equal to $\mu_n$, which  converges to $3/4$.

Before moving forward, we examine the assumptions of Proposition \ref{propmain} that underlie our primary findings.  The condition  $r_n=o(1)$ ensures that the resulting graph is relatively sparse. This assumption is consistent with empirical observations, as most real-world networks exhibit sparsity \cite{A17}. The condition \(f(x)=g(x)I[0\le x\le 1]\), where \(g\) satisfies \(g(x+1)=g(x)=g(x-1)\) for all \(x\in \mathbb{R}\), effectively defines   \(f(x)\) as a  density function on a circle with circumference 1.  This assumption is crucial for the analytical tractability of the leading terms in (\ref{prope0})--(\ref{defafbf1}) of Proposition \ref{propmain}. Without this periodic structure, the expressions for these terms would become prohibitively complex due to boundary effects. The assumption that \(g(x)\) possesses  a bounded fourth derivative is a technical requirement. This condition allows for a Taylor expansion that ensures the remainder terms  of  (\ref{prope0})--(\ref{defafbf1}) are $O(r_n^3)$ or \(O(r_{n}^{5})\) or \(O(r_{n}^{6})\). Crucially, these bounds hold uniformly in $X_1$
, ensuring the remainders are independent of the specific location of the vertex.

The assumptions imposed on the density \(f(x)\) in Proposition \ref{propmain} are relatively mild and easily satisfied in practice. Notably, they are fulfilled by the uniform distribution as well as the ubiquitous von Mises density. A more comprehensive analysis of these conditions is provided later. Equipped with Proposition \ref{propmain}, we now present a central limit theorem for the global clustering coefficient.

\begin{Theorem}\label{mthm}
Suppose the assumption of Proposition \ref{propmain} holds. Let $\mu_n=\frac{\mathbb{E}[A_{12}A_{13}A_{23}]}{\mathbb{E}[A_{12}A_{13}]}$. Then the following results hold. 
\vskip 2mm
(I). If $nr_n^5=\omega(1)$, then
\begin{equation}\label{thmexp2}
    \frac{16\sqrt{n}\mathbb{E}[f^2(X_1)]}{3r_n^2\sigma_1}\big(\mathcal{C}_n-\mu_n\big)\Rightarrow N(0,1),
\end{equation}
provided that $\sigma_1^2=\mathbb{E}\Big[\Big(-3c_ff^2(X_1)-2f(x)f^{\prime\prime}(X_1)-(f^{\prime}(X_1))^2\Big)^2\Big]>0$, where $c_f=\frac{\mathbb{E}[(f'(X_1))^2]}{\mathbb{E}[f^2(X_1)]}$. 

\vskip 2mm
(II). Suppose $nr_n=\omega(1)$. If $nr_n^5=o(1)$ or $f(x)$ is the uniform density, then
\begin{equation}\label{thmexp4}
    \frac{2\sqrt{2}nr_n^2\mathbb{E}[f^2(X_1)]}{3\sigma_{2n}}\big(\mathcal{C}_n-\mu_n\big)\Rightarrow N(0,1),
\end{equation}
where $\sigma_{2n}^2=\mathbb{E}[h(X_1,X_2,X_3)h(X_1,X_2,X_4)]=\Theta(r_n^3)$ and
\begin{align}\label{deofhf}
    h(X_1, X_2, X_3) = A_{12}A_{13}A_{23} - \frac{\mu_n}{3}(A_{12}A_{13} + A_{12}A_{23} + A_{13}A_{23}).
\end{align}
\vskip 2mm
(III). If $nr_n=o(1)$ and $n^3 r_n^2=\omega(1)$, then
\begin{equation}\label{thmexp6}
    \frac{8n\sqrt{n}r_n\sqrt{\mathbb{E}[f^2(X_1)]}}{3}\big(\mathcal{C}_n-\mu_n\big)\Rightarrow N(0,1).
\end{equation}
\end{Theorem}

Theorem \ref{mthm} establishes a central limit theorem for the global clustering coefficient of random geometric graphs, focusing particularly on the nonuniform case. After suitable centering and scaling, the global clustering coefficient converges in distribution to the standard normal distribution. When $n^3r_n^2=o(1)$, the expected number of triangles tends to zero according to Proposition \ref{propmain}. Consequently, this regime is excluded from the scope of Theorem \ref{mthm}.

The proof strategy for Theorem \ref{mthm} is structured as follows.
In Case (I), the dense regime, we show that the coefficient is asymptotically equivalent to a sum of i.i.d. terms, allowing us to apply the Lyapunov CLT (see Lemma \ref{varlem}). In Case (II), the intermediate dense case, we prove that the coefficient is asymptotically equivalent to a 
$U$-statistic of order 2 with a sample-size-dependent kernel. We then utilize the asymptotic theory for such statistics to derive the distribution (see Lemmas \ref{ustclt} and \ref{varlem2}). Finally, in Case (III), the sparse regime, since no existing theory directly applies as far as we know, we employ the method of moments to derive the asymptotic distribution (see Lemma \ref{varlem3}).

Theorem \ref{mthm} provides the exact limit of the global clustering coefficient; to the best of our knowledge, this value was previously unknown for non-uniform random geometric graphs.
According to Theorem  \ref{mthm}, the global clustering coefficient is asymptotically equal to $\mu_n$.  By Proposition \ref{propmain}, we have
\[\mathcal{C}_n=\frac{3+r_n^2a_f}{4+r_n^2b_f}+o_P(1).\]
For a uniform random geometric graph,  $f(x)$ is constant (specifically, $f(x)=1$),  which implies $a_f=b_f=0$; thus, the limit is precisely $3/4$, consistent with the result established in  \cite{Y25c}. For non-uniform random geometric graphs, the limit is $3/4$ with an additional error term of order $O(r_n^2)$. When $r_n=o(1)$, the limits are asymptotically identical.

Although the global clustering coefficients for both uniform and non-uniform RGGs share the same limit and asymptotic distribution, their convergence rates differ in the dense regime. Note that our result for non-uniform RGGs in Case (I) does not apply to the uniform case, as the latter corresponds to the setting where $\sigma_1=0$.
In the dense case where $nr_n^5=\omega(1)$, the clustering coefficient for non-uniform RGGs (with $\sigma_1>0$) converges to the standard normal distribution at a rate of $\frac{r_n^2}{\sqrt{n}}$. In contrast, for the uniform RGG, the convergence rate is $\frac{1}{n\sqrt{r_n}}$ (see Case (II)).  Since $\frac{1}{n\sqrt{r_n}}=o\left(\frac{r_n^2}{\sqrt{n}}\right)$ under the condition $nr_n^5=\omega(1)$, the convergence rate for the non-uniform case is notably slower than that of the uniform case. In the sparse regime where $nr_n^5=o(1)$, the convergence rates for uniform and non-uniform RGGs are identical. A similar phenomenon is observed in the asymptotic properties of the friendship paradox in \cite{Y24,Y26}. The paradox exhibits different limits for uniform and nonuniform random geometric graphs in the dense case, whereas these limits are identical in the sparse regime.

We now provide examples of density functions 
 satisfying the assumptions of Theorem \ref{mthm}. Beyond the obvious case of the uniform density function, we provide the following nontrivial example. Let us consider the family of von Mises densities on $[0,1]$, defined  by
    \begin{equation}\label{vonMises}
    f(x)=g(x)I[0\leq x\leq 1],\hskip 1cm g(x)=\frac{e^{\kappa \cos(2\pi x-\mu)}}{\mathcal{I}_0},
\end{equation}
    where $\kappa\geq0$, $\mu\in[0,1]$, and
    \[\mathcal{I}_0=\frac{1}{2\pi}\int_0^{2\pi}e^{\kappa \cos(x)}dx.\]
  The uniform distribution is recovered when $\kappa=0$. As the concentration parameter $\kappa$ increases, the distribution clusters more tightly around the mean $\mu$, eventually behaving like a normal distribution with variance $\frac{1}{\kappa}$.

The function $g(x)$ in (\ref{vonMises}) is smooth, with bounded and continuous derivatives of all orders. It is strictly bounded from below by the positive constant $(e^{\kappa}I_0(\kappa))^{-1}$,  ensuring the density is non-vanishing on its domain.  Furthermore, the periodicity of the cosine function implies that $g(x)$
 satisfies the periodic boundary conditions $g(x+1)=g(x)=g(x-1)$ for all $x\in\mathbb{R}$. Consequently, the von Mises distribution satisfies the assumptions of Theorem \ref{mthm}. The parameter $\sigma_1^2$ in Case (I) does not admit a simple closed-form expression. We provide several numerical values corresponding to different sets of parameters $(\kappa,\mu)$ in Table \ref{tab1}. All values of $\sigma_1^2$ are positive. A larger $\kappa$
 results in a larger $\sigma_1^2$, while the values remain invariant with respect to the mean $\mu$. A possible reason for this is that $f(x)$ 
 is a probability density on a circle and the distance is translation-invariant.

\begin{table}[!h]
\begin{center}
\caption{Numeric values of $\sigma_1^2$. }
\label{tab1}
\begin{tabular}{ |c| c| c|c| c| } 
\hline
$(\kappa,\ \mu)$ & (0.1,\ 0.1) & (0.5,\ 0.1) & (1.0,\ 0.1) & (5,\ 0.1)  \\
 \hline
$\sigma_1^2$ & 31.78  & 1264.83 & 13924.35  & 13646828.67    \\ 
 \hline  
\hline
$(\kappa,\ \mu)$ & (0.1,\ 0.3) & (0.5,\ 0.3) & (1.0,\ 0.3) & (5,\ 0.3)  \\
 \hline
$\sigma_1^2$ & 31.78  & 1264.83 & 13924.35  & 13646828.67\\  
 \hline  
 \hline
$(\kappa,\ \mu)$ & (0.1,\ 0.5) & (0.5,\ 0.5) & (1.0,\ 0.5) & (5,\ 0.5)  \\
 \hline
$\sigma_1^2$ & 31.78  & 1264.83 & 13924.35  & 13646828.67 
\\ 
 \hline
\end{tabular}
\end{center}
\end{table}

\section{Proof of main results}\label{secpf}
In this section, we provide detailed proofs of the main results.

\subsection{ Proof of Proposition \ref{propmain}}

Firstly, we prove (\ref{prope0}). Note that $d(X_1,X_2)\leq r_n$ is equivalent to three cases: (a) $r_n\leq X_1\leq 1-r_n$ and $X_1-r_n\leq X_2\leq X_1+r_n$, (b) $0\leq X_1\leq r_n$ and $0\leq X_2\leq X_1+r_n$ or $1-r_n+X_1\leq X_2\leq 1$, (c) $1-r_n\leq X_1\leq 1$ and $X_1-r_n\leq X_2\leq 1$ or $0\leq X_2\leq X_1+r_n-1$. Then the conditional expectation $\mathbb{E}[A_{12}|X_1]$ can be expressed as
\begin{equation*}
\mathbb{E}[A_{12}|X_1]=
        \begin{cases}
      \int_{X_1 - r_n}^{X_1 + r_n} f(x) dx, & \text{if  $r_n\leq X_1\leq 1-r_n$,}\\
      \int_{0}^{X_1 + r_n} f(x) dx+\int_{1- r_n+X_1}^{1} f(x)dx, & \text{if $0\leq X_1\leq r_n$,}\\
      \int_{X_1-r_n}^{1} f(x) dx+\int_{0}^{X_1+r_n-1} f(x)dx, & \text{if $1-r_n\leq X_1\leq 1$. }
    \end{cases} 
\end{equation*}
By assumption, $f(x)=g(x)I[0\leq x\leq 1]$, and $g(x)$ is a  periodic function with period one. Then $g(x+1)=g(x)=g(x-1)$ for all $x\in \mathbb{R}$. Using a change of variables in the definite integral, one has
\begin{align*}
\int_{1- r_n+X_1}^{1} f(x)dx&=\int_{1- r_n+X_1}^{1} g(x)I[0\leq x\leq1]dx\\
&=\int_{X_1-r_n}^{0} g(x+1)I[0\leq x+1\leq1] dx\\
&=\int_{X_1-r_n}^{0} g(x) dx,
\end{align*}
and
\begin{align*}
    \int_{0}^{X_1+r_n-1} f(x)dx&=\int_{1}^{X_1+r_n} g(x-1)I[0\leq x-1\leq1]dx\\
    &=\int_{1}^{X_1+r_n} g(x)dx.
\end{align*}
Hence, for all $X_1\in[0,1]$, we obtain
\begin{align*} 
    \mathbb{E}[A_{12}|X_1]&=\int_{X_1 - r_n}^{X_1 + r_n} g(x) dx.
\end{align*}
By assumption,  $g(x)$ has bounded fourth derivative. Moreover, $|x-X_1|\leq r_n$ for $X_1-r_n\leq x\leq X_1+r_n$. Using a fourth-order Taylor expansion of $g(x)$ at $X_1$, we have
\begin{align}\label{taylor1}
g(x)=\sum_{k=0}^3\frac{g^{(k)}(X_1)}{k!} (x-X_1)^k + O(r_n^4),
\end{align}
where the remainder term $O(r_n^4)$ does not depend on $X_1$. Based on (\ref{taylor1}), it is easy to obtain
\begin{align} \label{frieq1}
\int_{X_1-r_n}^{X_1 + r_n} g(x) dx  
 &= 2r_n g(X_1) + \frac{g''(X_1)}{3} r_n^3+ O(r_n^5).
\end{align}
Note that  $f(x)=g(x)I[0\leq x\leq 1]$. For $X_1\in[0,1]$, we have $f(X_1)=g(X_1)$ and $f''(X_1)=g''(X_1)$.
Thus, equation (\ref{prope0}) is proved.

{\bf Proof of equation (\ref{prope1})}. The equation (\ref{prope1}) is trivial, noting that  (\ref{prope0}) holds and 
\[\mathbb{E}[A_{12}A_{13}|X_1]=\mathbb{E}[A_{12}|X_1]\mathbb{E}[A_{13}|X_1]=\big(\mathbb{E}[A_{12}|X_1]\big)^2.\]

{\bf Proof of equation (\ref{prope01})}.  By the properties of conditional expectations and  equation (\ref{prope0}), we have
\begin{align}\nonumber
   \mathbb{E}[A_{12}A_{23}|X_1]&= \mathbb{E}[\mathbb{E}[A_{12}A_{23}|X_1,X_2]|X_1]\\  \nonumber
   &=\mathbb{E}[A_{12}\mathbb{E}[A_{23}|X_2]|X_1]\\ \label{pfeq1}
   &=2r_n\mathbb{E}[A_{12} f(X_2)|X_1] +  \frac{r_n^3}{3}\mathbb{E}[A_{12}f''(X_2)|X_1]+ O(r_n^5)\mathbb{E}[A_{12}|X_1].
\end{align}
By a similar argument as in the proof of equation (\ref{prope0}), we get
\begin{align}\nonumber
\mathbb{E}[A_{12} f(X_2)|X_1]&=\int_{X_1-r_n}^{X_1+r_n}g^2(x)dx\\ \nonumber
&=2r_ng^2(X_1)+\frac{2}{3}\big((g'(X_1))^2+g(X_1)g''(X_1)\big)r_n^3+O(r_n^5)\\ \label{pfeq2}
&=2r_nf^2(X_1)+\frac{2}{3}\big((f'(X_1))^2+f(X_1)f''(X_1)\big)r_n^3+O(r_n^5),
\end{align}
and
\begin{align}\label{pfeq3}
\mathbb{E}[A_{12}f''(X_2)|X_1]&=\int_{X_1-r_n}^{X_1+r_n}g (x)g''(x)dx=2r_ng (X_1)g''(X_1)+O(r_n^3).
\end{align}

Combining (\ref{pfeq1})-(\ref{pfeq3}) yields (\ref{prope01}).

{\bf Proof of equation (\ref{prope3}):}  We prove (\ref{prope3}) in three scenarios: (I) $X_1\in(r_n,1-r_n)$; (II) $X_1\in[0,r_n)$; and (III) $X_1\in[1-r_n,1]$.

Firstly, we consider $X_1\in(r_n,1-r_n)$. In this case, \(A_{12}A_{13}A_{23}=1\) if and only if one of the following conditions holds: (a) $X_2\in(X_1,X_1+r_n)$ and $X_3\in(X_1,X_1+r_n)$ or $X_3\in(X_2-r_n,X_1)$; (b) $X_2\in(X_1-r_n,X_1)$ and $X_3\in(X_1-r_n,X_1)$ or $X_3\in(X_1,X_2+r_n)$. Then the conditional expectation in (\ref{prope3})  can be expressed as
\begin{eqnarray}\nonumber
    \mathbb{E}[A_{12}A_{13}A_{23}|X_1]  &=&\int_{X_1}^{X_1+r_n}f(x_2)dx_2\left(\int_{X_1}^{X_1+r_n}f(x_3)dx_3+\int_{X_2-r_n}^{X_1}f(x_3)dx_3\right)\\ \label{febpfeq1}
    &&+\int_{X_1-r_n}^{X_1}f(x_2)dx_2\left(\int_{X_1-r_n}^{X_1}f(x_3)dx_3+\int_{X_1}^{X_2+r_n}f(x_3)dx_3\right)
\end{eqnarray}
The second term of (\ref{febpfeq1}) can be obtained from the first term by changing $r_n$ to $-r_n$. Therefore, we only need to calculate the first term of (\ref{febpfeq1}).

Since $f(x)$
 is assumed to have a bounded fourth derivative, Taylor expanding $f(x)$
 about $X_1$ yields
\begin{align}\label{febpfeq2}
\int_{X_1}^{X_1 + r_n} f(x_2) dx_2=f(X_1) r_n + \frac{f'(X_1)}{2} r_n^2 + \frac{f''(X_1)}{6} r_n^3 + \frac{f'''(X_1)}{4!} r_n^4 + O(r_n^5).
\end{align}
Then
\begin{align}\nonumber
&\left(\int_{X_1}^{X_1 + r_n} f(x_2) dx_2\right)^2 \\ \label{febpfeq3}
&= f(X_1)^2 r_n^2 + \frac{(f'(X_1))^2}{4} r_n^4 + f(X_1) f'(X_1) r_n^3 + \frac{2f(X_1) f''(X_1)}{3!} r_n^4 + O(r_n^5).
\end{align}
Let $F(t)=\int_0^tf(x)dx$. It is straightforward to get that
\begin{align}\label{febpfeq4}
\int_{X_1}^{X_1+r}f(x_2)dx_2\int_{X_2-r_n}^{X_1}f(x_3)dx_3&= \int_{X_1}^{X_1 + r_n} F(X_1) f(x_2) dx_2 - \int_{X_1}^{X_1 + r_n} F(x_2 - r_n) f(x_2) dx_2.
\end{align}
By (\ref{febpfeq2}), we have
\begin{align}\label{febpfeq5}
\int_{X_1}^{X_1 + r_n} F(X_1) f(x_2) dx_2=F(X_1) \left[f(X_1) r_n + \frac{f'(X_1)}{2} r_n^2 + \frac{f''(X_1)}{6} r_n^3 + \frac{f'''(X_1)}{4!} r_n^4 + O(r_n^5)\right].
\end{align}
Now we evaluate the second integral in (\ref{febpfeq4}). 
Applying the Taylor expansion to $F(x_2 - r_n)$ at $x_2$ yields
\begin{align}\label{febpfeq6}
F(x_2 - r_n) f(x_2) &= \left(F(x_2) - f(x_2) r_n + \frac{f'(x_2)}{2!} r_n^2 - \frac{f''(x_2)}{3!} r_n^3 + \frac{f'''(x_2)}{4!} r_n^4+O(r_n^5)\right) f(x_2).
\end{align}

The integral of the first term of (\ref{febpfeq6}) is equal to
\begin{align}\label{febpfeq7}
\int_{X_1}^{X_1 + r_n} F(x_2) f(x_2) dx_2 &=\frac{1}{2} \left[F(X_1 + r_n)^2 - F(X_1)^2\right].
\end{align}
By the Taylor expansion of $F(X_1 + r_n)$ at $X_1$, we have
\begin{align*}
F(X_1 + r_n)=F(X_1)+f(X_1)r_n+\frac{1}{2}f'(X_1)r_n^2+\frac{1}{3!}f''(X_1)r_n^3+\frac{1}{4!}f'''(X_1)r_n^4+O(r_n^5).
\end{align*}
Then we have
\begin{align}\nonumber
\int_{X_1}^{X_1 + r_n} F(x_2) f(x_2) dx_2 &= \frac{1}{2} \bigg[ f(X_1)^2 r_n^2 + \frac{(f'(X_1))^2}{4} r_n^4 + 2 F(X_1) f(X_1) r_n + F(X_1) f'(X_1) r_n^2\\ \nonumber
&\quad + \frac{2 \cdot F(X_1) f''(X_1)}{3!} r_n^3 + \frac{2 F(x_1) f'''(X_1)}{4!} r_n^4 + f(X_1) f'(X_1) r_n^3\\ \label{febpfeq8}
& + \frac{2 f(X_1) f''(X_1)}{3!} r_n^4+ O(r_n^5)\bigg].
\end{align}

Substituting the Taylor expansion of $f(x)$, the integral of the second term of (\ref{febpfeq6}) is equal to
\begin{align}\nonumber
&r_n \int_{X_1}^{X_1 + r_n} f(x_2)^2 dx_2 \\ \nonumber
&= r_n \int_{X_1}^{X_1 + r_n} \Big[f(X_1) + f'(X_1)(x_2 - X_1) + \frac{f''(X_1)}{2!}(x_2 - X_1)^2 \\ \nonumber
&\quad+ \frac{f'''(X_1)}{3!}(x_2 - X_1)^3 + O(r_n^4)\Big]^2 dx_2 \\ \nonumber
& = r_n \int_{X_1}^{X_1 + r_n} \Big[f(X_1)^2 + (f'(X_1))^2(x_2 - X_1)^2 + 2 f(X_1) f'(X_1)(x_2 - X_1)\\ \nonumber
&\quad+ f(X_1) f''(X_1)(x_2 - X_1)^2\Big] dx_2+O(r_n^5)\\ \label{febpfeq9}
&= f(X_1)^2 r_n^2 + \frac{(f'(X_1))^2}{3} r_n^4 + f(X_1) f''(X_1) \frac{r_n^4}{3} + f(X_1) f'(X_1) r_n^3 + O(r_n^5).
\end{align}

Similarly, the integral of the third term of (\ref{febpfeq6}) is equal to
\begin{align}\nonumber
\frac{r_n^2}{2!} \int_{X_1}^{X_1 + r_n} f(x_2) f'(x_2) dx_2&= \frac{r_n^2}{4} \left[f(X_1 + r_n)^2 - f(X_1)^2\right]\\ \nonumber
&= \frac{r_n^2}{4} \left[\left(f(X_1) + f'(X_1) r_n + \frac{f''(X_1)}{2} r_n^2\right)^2 - f(X_1)^2\right] + O(r_n^5) \\ \label{febpfeq10}
&= \frac{r_n^2}{4} \left[ (f'(X_1))^2 r_n^2 + 2 f(X_1) f'(X_1) r_n + f(X_1) f''(X_1) r_n^2\right]+O(r_n^5).
\end{align}

The integral of the fourth term of (\ref{febpfeq6}) is equal to
\begin{align} \label{febpfeq11}
\frac{r_n^3}{3!} \int_{X_1}^{X_1 + r_n} f''(x_2) f(x_2) dx_2
&= \frac{f''(X_1) f(X_1)}{3!} r_n^4 + O(r_n^5).
\end{align}

Combining (\ref{febpfeq4})-(\ref{febpfeq11}) yields
\begin{align}\nonumber
&\int_{X_1}^{X_1 + r_n} f(x_2) dx_2 \int_{X_2-r_n}^{X_1} f(x_3) dx_3 \\  \label{febpfeq12}
&= 
 \frac{r_n^2}{2} f^2(X_1)  + \left[\frac{f(X_1) f''(X_1)}{12} - \frac{(f'(X_1))^2}{24}\right] r_n^4 + O(r_n^5).
\end{align}

By (\ref{febpfeq3}) and (\ref{febpfeq12}), we have
\begin{align}\nonumber
&\int_{X_1}^{X_1 + r_n} f(x_2) dx_2 \int_{X_2 - r_n}^{X_1 + r_n} f(x_3) dx_3\\ \label{febpfeq13}
& = \frac{3}{2} f(X_1)^2 r_n^2 + \frac{5}{24} \left[(f'(X_1))^2 + 2 f(X_1) f''(X_1)\right] r_n^4 + f(X_1) f'(X_1) r_n^3 + O(r_n^5).
\end{align}

Changing $r_n$ of the first term of (\ref{febpfeq1}) to $-r_n$, we get the second term of (\ref{febpfeq1}). That is,
\begin{eqnarray}\nonumber &&\int_{X_1-r_n}^{X_1}f(x_2)dx_2\left(\int_{X_1-r}^{X_1}f(x_3)dx_3+\int_{X_1}^{X_2+r_n}f(x_3)dx_3\right)\\ \label{febpfeq14}
&=& \frac{3}{2} f(X_1)^2 r_n^2 + \frac{5}{24} \left[(f'(X_1))^2 + 2 f(X_1) f''(X_1)\right] r_n^4 - f(X_1) f'(X_1) r_n^3 + O(r_n^5).
\end{eqnarray}
In view of (\ref{febpfeq1}), (\ref{febpfeq13}) and (\ref{febpfeq14}), it follows that for $r_n<X_1\leq 1-r_n$:  
\begin{eqnarray}\nonumber \mathbb{E}[A_{12}A_{13}A_{23}|X_1]= 3r_n^2f^2(X_1)  + \frac{5r_n^4}{12} \left[(f'(X_1))^2 + 2 f(X_1) f''(X_1)\right] + O(r_n^5).
\end{eqnarray}

Now we consider the case $X_1\in[0,r_n)$. In this case, \(A_{12}A_{13}A_{23}=1\) is satisfied if and only if one of the following four configurations occurs: If $X_2\in[0,X_1]$, then $X_3\in(X_1,X_2+r_n)$ or $X_3\in(0,X_1)$ or $X_3\in(1+X_1-r_n,1)$. If $X_2\in(1+X_1-r_n,1)$, then $X_3\in[0,X_1]$ or $X_3\in[1+X_1-r_n,1]$ or $X_3\in[X_1,X_2+r_n-1]$. If $X_2\in[X_1,r_n]$, then $X_3\in[0,X_1+r_n]$ or $X_3\in[X_2+1-r_n,1]$. If $X_2\in[r_n,X_1+r_n]$, then $X_3\in(X_1,X_1+r_n)$ or $X_3\in(X_2-r_n,X_1)$. Then the conditional expectation $\mathbb{E}[A_{12}A_{13}A_{23}|X_1]$ is written as
\begin{align*}
\mathbb{E}[A_{12}A_{13}A_{23}|X_1]  
&=\int_{0}^{X_1}f(x_2)dx_2\left(\int_{0}^{x_2+r_n}f(x_3)dx_3+\int_{1+X_1-r_n}^{1}f(x_3)dx_3\right)\\
&\quad+\int_{1+X_1-r_n}^{1}f(x_2)dx_2\left(\int_{0}^{x_2+r_n-1}f(x_3)dx_3+\int_{1+X_1-r_n}^{1}f(x_3)dx_3\right)\\  \nonumber
&\quad+\int_{X_1}^{r_n}f(x_2)dx_2\left(\int_0^{X_1+r_n}f(x_3)dX_3+\int_{X_2+1-r_n}^1f(x_3)dX_3\right)\\  \nonumber
    &\quad+\int_{r_n}^{X_1+r_n}f(x_2)dx_2\int_{x_2-r_n}^{X_1+r_n}f(x_3)dx_3.
\end{align*}

By assumption, $f(x)=g(x)I[0\leq x\leq 1]$, and \(g(x)\)  satisfies \(g(x+1)=g(x)=g(x-1)\) for all \(x\in \mathbb{R}\). Applying a change of variables to the definite integrals, we obtain
\begin{align*}
\int_{1+X_1-r_n}^{1}f(x_3)dx_3&=\int_{X_1-r_n}^{0}g(y+1)dy=\int_{X_1-r_n}^{0}g(y)dy\\
\int_{x_2+1-r_n}^1f(x_3)dx_3&=\int_{x_2-r_n}^{0}g(y+1)dy=\int_{x_2-r_n}^{0}g(y)dy,
\end{align*}
\begin{align*}
\int_{1+X_1-r_n}^{1}f(x_2) \int_{X_1-r_n}^{x_2+r_n-1}f(x_3)dx_3dx_2&=\int_{X_1-r_n}^{0}g(x_2+1) \int_{X_1-r_n}^{x_2+r_n}g(x_3)dx_3dx_2\\
&=\int_{X_1-r_n}^{0}g(x_2) \int_{X_1-r_n}^{x_2+r_n}g(x_3)dx_3dx_2.
\end{align*}
Therefore, we have
\begin{align*}
\mathbb{E}[A_{12}A_{13}A_{23}|X_1]  
&=\int_{0}^{X_1}f(x_2)dx_2\int_{X_1-r_n}^{x_2+r_n}f(x_3)dx_3+\\
&\quad+\int_{X_1-r_n}^{0}f(x_2) \int_{X_1-r_n}^{x_2+r_n}f(x_3)dx_3dx_2\\  \nonumber
    &\quad+\int_{X_1}^{X_1+r_n}g(x_2)dx_2\int_{x_2-r_n}^{X_1+r_n}g(x_3)dx_3\\
    &=\int_{X_1-r_n}^{X_1}g(x_2)dx_2\int_{X_1-r_n}^{x_2+r_n}g(x_3)dx_3\\\nonumber
    &\quad+\int_{X_1}^{X_1+r_n}g(x_2)dx_2\int_{x_2-r_n}^{X_1+r_n}g(x_3)dx_3,
\end{align*}
By the same argument as used in (\ref{febpfeq1}), we conclude (\ref{prope3}) holds for $X_1\in[0,r_n)$.

For $X_1\in[1-r_n,1]$, it is easy to show that (\ref{prope3}) holds by a similar argument.\\

{\bf Proof of equations (\ref{prope4}) and (\ref{prope2}): } Equations (\ref{prope4}) and (\ref{prope2}) follow directly from (\ref{prope1}) and (\ref{prope3}), respectively.\\

{\bf Proof of equation (\ref{defafbf1}): } By (\ref{prope4}) and (\ref{prope2}), it is easy to verify that
\begin{align*}
    \mu_n&=\frac{3r_n^2\int_0^1f^3(x)dx+ \frac{5r_n^4}{12} \int_0^1f(x)\left[(f'(x))^2 + 2f(x) f''(x)\right]dx  + O(r_n^5)}{4 r_n^2\int_0^1f^3(x)dx  + \frac{4r_n^4 }{3} \int_0^1f^2(x) f''(x)dx + O(r_n^5)}\\
    &=\frac{3+r_n^2a_f}{4+r_n^2b_f}+O(r_n^3),
\end{align*}
Then the proof is complete.

\qed

\subsection{Some lemmas}
Before proving Theorem \ref{mthm}, we present several lemmas.
For convenience, let $\Delta_{123}=A_{12}A_{23}A_{31}$ and $P_{123}=A_{12}A_{23}$. Recall that $\mu_n=\frac{\mathbb{E}[\Delta_{123}]}{\mathbb{E}[P_{123}]}$. The symmetric function $ h(X_1, X_2, X_3)$ defined in (\ref{deofhf}) is equal to
\begin{align}\label{defhf}
    h(X_1, X_2, X_3) = \Delta_{123} - \frac{\mu_n}{3}(P_{123} + P_{213} + P_{231}).
\end{align}
Let 
\begin{align}\label{h1h2xx}
  h_1(X_1) = \mathbb{E}[h(X_1, X_2, X_3) | X_1],\ \ \ \ \ \  h_2(X_1, X_2) = \mathbb{E}[h(X_1, X_2, X_3) | X_1, X_2]. 
\end{align}

First, we provide the asymptotic count of 2-paths associated with the global clustering coefficient.
\begin{Lemma}\label{00varlem}
Suppose the assumption of Proposition \ref{propmain} holds and $n^3r_n^2=\omega(1)$. Then we have
\[\frac{\sum_{i\neq j\neq k}A_{ij}A_{jk}}{4n^3r_n^2\mathbb{E}[f^2(X_1)]}=1+o_P(1).\]
\end{Lemma}

{\bf Proof of Lemma \ref{00varlem}:} By Proposition \ref{propmain}, $\mathbb{E}[A_{12}A_{23}]=4r_n^2\mathbb{E}[f^2(X_1)]+O(r_n^4)$. The proof proceeds by showing that $\sum_{i\neq j\neq k}A_{ij}A_{jk}$ is asymptotically equal to $n^3\mathbb{E}[A_{12}A_{23}]$. To this end, we show that the variance of $\sum_{i\neq j\neq k}A_{ij}A_{jk}$ is of smaller order than $n^6(\mathbb{E}[A_{12}A_{23}])^2$. The variance can be expressed as
\begin{align}\nonumber
    &\mathbb{E}\left[\left(\sum_{i\neq j\neq k}(A_{ij}A_{jk}-\mathbb{E}[A_{ij}A_{jk}])\right)^2\right]\\ \label{ervea1}
    &=\sum_{\substack{i\neq j\neq k\\ i_1\neq j_1\neq k_1}}\mathbb{E}\left[(A_{ij}A_{jk}-\mathbb{E}[A_{ij}A_{jk}])(A_{i_1j_1}A_{j_1k_1}-\mathbb{E}[A_{ij}A_{jk}])\right].
\end{align}
If $\{i,j,k\}\cap\{i_1,j_1,k_1\}=\emptyset$, then $A_{ij}A_{jk}$ and $A_{i_1j_1}A_{j_1k_1}$ are independent. In this case, the expectation in (\ref{ervea1}) vanishes. Then $\{i,j,k\}\cap\{i_1,j_1,k_1\}\neq\emptyset$. Suppose  $|\{i,j,k\}\cap\{i_1,j_1,k_1\}|=1$. There are at most $n^5$ such indices. In the case where $i=i_1$, it follows from Proposition \ref{propmain} that
\begin{align*}
&\mathbb{E}\left[(A_{ij}A_{jk}-\mathbb{E}[A_{ij}A_{jk}])(A_{i_1j_1}A_{j_1k_1}-\mathbb{E}[A_{ij}A_{jk}])\right]\\
&=\mathbb{E}\left[A_{ij}A_{jk}A_{ij_1}A_{j_1k_1}\right]-\mathbb{E}[A_{ij}A_{jk}]\mathbb{E}[A_{ij}A_{jk}]\\
&=\mathbb{E}\left[\mathbb{E}\left[A_{ij}A_{jk}|X_i\right]\mathbb{E}\left[A_{ij_1}A_{j_1k_1}|X_i\right]\right]-\mathbb{E}[A_{ij}A_{jk}]\mathbb{E}[A_{ij}A_{jk}]\\
&=O(r_n^4).
\end{align*}
The remaining cases, such as  $i=j_1$ and $i=k_1$,  can be bounded similarly.  Suppose  $|\{i,j,k\}\cap\{i_1,j_1,k_1\}|=2$. There are at most $n^4$ such indices.  If $i_1,j_1\in \{i,j,k\}$, then
\begin{align*}
&|\mathbb{E}\left[(A_{ij}A_{jk}-\mathbb{E}[A_{ij}A_{jk}])(A_{i_1j_1}A_{j_1k_1}-\mathbb{E}[A_{ij}A_{jk}])\right]|\\
&\leq\mathbb{E}\left[A_{ij}A_{jk}A_{j_1k_1}\right]+\mathbb{E}[A_{ij}A_{jk}]\mathbb{E}[A_{ij}A_{jk}]\\
&=O(r_n^3).
\end{align*}
The remaining cases $i_1,k_1\in \{i,j,k\}$ and $j_1,k_1\in \{i,j,k\}$  can be bounded similarly.  Suppose  $\{i,j,k\}=\{i_1,j_1,k_1\}$. There are at most $n^3$ such indices. In this case, 
\begin{align*}
&|\mathbb{E}\left[(A_{ij}A_{jk}-\mathbb{E}[A_{ij}A_{jk}])(A_{i_1j_1}A_{j_1k_1}-\mathbb{E}[A_{ij}A_{jk}])\right]|\\
&\leq\mathbb{E}\left[A_{ij}A_{jk}\right]+\mathbb{E}[A_{ij}A_{jk}]\mathbb{E}[A_{ij}A_{jk}]\\
&=O(r_n^2).
\end{align*}
In summary, we have
\begin{align*}
    \mathbb{E}\left[\left(\sum_{i\neq j\neq k}(A_{ij}A_{jk}-\mathbb{E}[A_{ij}A_{jk}])\right)^2\right]=O(n^5r_n^4+n^4r_n^3+n^3r_n^2)=o(n^6r_n^4).
\end{align*}
Then the result of Lemma \ref{00varlem} follows from  Markov's inequality. 

\qed

Next, we present an asymptotic result for 
$U$-statistics of order 2 with sample-size-dependent kernels.
Let $U_n=\frac{1}{2}\sum_{i\neq j}k_n(X_i,X_j)$, where $k_n(x,y)$ is a symmetric function. Without loss of generality, let $\mathbb{E}[k_n(X_i,X_j)]=0$. Let $q_n(x)=\mathbb{E}[k_n(x,X_1)]$. Denote $v_n^2=\frac{n^2}{2}\mathbb{E}[k_n^2(X_1,X_2)]+n^3\mathbb{E}[q_n^2(X_1)]$.The following result from \cite{JJ86} characterizes the asymptotic distribution of this $U_n$.

\begin{Lemma}\label{ustclt}
Suppose the following conditions holds.
\begin{align}\label{ustclt1}
    \sup_{x,y}|k_n(x,y)|&=o(v_n),\\ \label{ustclt2}
    \sup_{x}\mathbb{E}[|k_n(x,X_1)|]&=o\left(\frac{v_n}{n}\right).
\end{align}
Then
\[\frac{U_n}{v_n}\Rightarrow N(0,1).\]
\end{Lemma}
Lemma \ref{ustclt} is instrumental in establishing the asymptotic distribution of the global clustering coefficient for nonuniform random geometric graphs in the intermediate dense regime.

Next, we provide several lemmas that establish the order of the variances and the asymptotic distributions of $h_1(X_1)$,  $h_2(X_1,X_2)$ and  $h(X_1, X_2, X_3)$ defined in (\ref{defhf}) and (\ref{h1h2xx}).

\begin{Lemma}\label{varlem}
    Under the assumption of Proposition \ref{propmain}, we have
    \begin{align}\label{varlempf1}
        \mathbb{E}[h_1^2(X_1)]= \frac{r_n^8}{16}\sigma_1^2+O(r_n^9),
    \end{align}
    where $\sigma_1^2=\mathbb{E}\big[\big(-3c_ff^2(X_1)-2f(X_1)f^{\prime\prime}(X_1)-(f^{\prime}(X_1))^2\big)^2\big]$.
    If $\sigma_1^2>0$, then 
   \begin{align}\label{1varlem}
       \frac{4\sum_ih_1(X_i)}{\sqrt{n}r_n^4\sigma_1}\Rightarrow N(0,1).
   \end{align}
\end{Lemma}

{\bf Proof of Lemma \ref{varlem}: } By Proposition \ref{propmain}, the function $h_1(X_1)$ has the following asymptotic expression
\begin{align*}
    h_1(X_1) &=\mathbb{E}[\Delta_{123}|X_1] - \frac{\mu_n}{3}(\mathbb{E}[P_{123}|X_1]+ \mathbb{E}[P_{213}|X_1] + \mathbb{E}[P_{231}|X_1])\\
    &=3r_n^2f^2(X_1)+ \frac{5 r_n^4}{12} \left[(f'(X_1))^2 + 2f(X_1) f''(X_1)\right]\\
    &\quad-\frac{3+r_n^2a_f}{3(4+r_n^2b_f)}\Big(12r_n^2f^2(X_1)+\frac{4r_n^4 }{3} f(X_1) f''(X_1)\\
    &\quad+\frac{r_n^4}{3}\big[8(f^{\prime}(X_1))^2+12f(X_1)f^{\prime\prime}(X_1)\big]\Big) + O(r_n^5)\\
    &=\frac{r_n^4}{4}\Big((3b_f-4a_f)f^2(X_1)-2f(X_1)f^{\prime\prime}(X_1)-(f^{\prime}(X_1))^2\Big)+O(r_n^5).
\end{align*}
Next, we simplify the term $3b_f-4a_f$. By assumption, $f(1)=g(1)=g(0)=f(0)$ and $f'(1)=g'(1)=g'(0)=f'(0)$. Then
\begin{align*}
\int_0^1f^2(x)f''(x)dx=f^2(x)f'(x)|_0^1-2\int_0^1f(x)(f'(x))^2dx=-2\int_0^1f(x)(f'(x))^2dx.
\end{align*}
Recall the definition of $a_f$ and $b_f$ in (\ref{defafbf}). Then
\begin{align*}
&(3b_f-4a_f)\int_0^1f^3(x)dx\\
&= 4\int_0^1f^2(x)f''(x)dx-\frac{5}{3}\int_0^1f(x)(f'(x))^2dx-\frac{10}{3}\int_0^1f^2(x)f''(x)dx\\
&=-3\int_0^1f(x)(f'(x))^2dx.
\end{align*}
Then it follows that
\begin{align}\label{2v1arlem}
    h_1(X_1) &=\frac{r_n^4}{4}\Big(-3c_ff^2(X_1)-2f(X_1)f^{\prime\prime}(X_1)-(f^{\prime}(X_1))^2\Big)+O(r_n^5).
\end{align}
Then we get (\ref{varlempf1}).

Next, we establish (\ref{1varlem}) by the Lyapunov Central Limit Theorem. By the definition of $h_1(X_1)$ in (\ref{h1h2xx}), it follows that $\mathbb{E}[h_1(X_1)]=0$. In view of (\ref{varlempf1}) and (\ref{2v1arlem}), it is clear that
\[s_n^2=\sum_i\mathbb{E}\left[\left(\frac{4h_1(X_i)}{\sqrt{n}r_n^4\sigma_1}\right)^2\right]=1+o(1).\]
Furthermore, according to (\ref{2v1arlem}), one can easily verify that 
   \begin{align*}
       \mathbb{E}\left[\sum_i\left(\frac{4h_1(X_i)}{\sqrt{n}r_n^4\sigma_1}\right)^4\right]=O\left(\frac{nr_n^{16}}{n^2r_n^{16}}\right)=O\left(\frac{1}{n}\right)=o(s_n^4).
   \end{align*}
Then  (\ref{1varlem}) follows from the Lyapunov Central Limit Theorem.

\qed

\begin{Lemma}\label{varlem2}
    Under the assumption of Proposition \ref{propmain}, we have
   \begin{align}\label{m1edeq2}
       \mathbb{E}[h_2^2(X_1,X_2)]= \Theta(r_n^3).
   \end{align} 
Suppose $nr_n=\omega(1)$. If $f(x)$ is the uniform density or $nr_n^5=o(1)$, then we have
      \begin{align}\label{m1edeq1}
       \frac{\sqrt{2}\sum_{i\neq j}h_2(X_i,X_j)}{2n\sigma_{2n}}\Rightarrow N(0,1),
   \end{align} 
   where $\sigma_{2n}^2=\mathbb{E}[h_2^2(X_1,X_2)]$.
\end{Lemma}

{\bf Proof of Lemma \ref{varlem2}: } First, we prove (\ref{m1edeq2}).  To begin with, we find a lower bound for $\mathbb{E}[h_2^2(X_1,X_2)]$.
Define event  $E=\{X_1\in(0.5,0.5+0.1r_n), X_2\in(0.5+1.2r_n,0.5+1.4r_n)\}$.  On the event $E$, $A_{12}=0$. Then $\Delta_{123}I[E]=P_{123}I[E]=P_{213}I[E]=0$.
In this case, by the definition of $h(X_1,X_2,X_3)$ in (\ref{defhf}), the function $h_2(X_1,X_2)$ can be expressed as
\begin{eqnarray}\nonumber
h_2(X_1,X_2)&=&\mathbb{E}[h(X_1,X_2,X_3)|X_1,X_2]I[E]+\mathbb{E}[h(X_1,X_2,X_3)|X_1,X_2]I[E^c]\\ \label{sum2p21}
&=&-\frac{\mu_n}{3}A_{13}A_{23}I[E]+\mathbb{E}[h(X_1,X_2,X_3)|X_1,X_2]I[E^c].
\end{eqnarray}
Note that $I[E]I[E^c]=0$. On the event $E$, if $X_3\in(0.5+0.8r_n,0.5+0.9r_n)$, then $A_{13}A_{23}=1$. Therefore, we have
\begin{eqnarray}\nonumber
h_2^2(X_1,X_2)&\geq\frac{\mu_n^2}{9}A_{13}A_{23}I[E]\geq \frac{\mu_n^2}{9}I[E]I[X_3\in(0.5+0.8r_n,0.5+0.9r_n)].
\end{eqnarray}
By assumption, $f(x)$ is bounded away from zero. There exists positive constant $c$ such that $f(x)\geq c>0$. Then
\begin{align}\nonumber
\mathbb{E}\left[h_2^2(X_1,X_2)\right]&\geq \frac{\mu_n^2}{9}\mathbb{E}\big[ I[E]I[X_3\in(0.5+0.8r_n,0.5+0.9r_n)]\big]\\ \nonumber
&\geq c^3\frac{\mu_n^2}{9} \int_{0.5}^{0.5+0.1r_n}dx_1\int_{0.5+1.2r_n}^{0.5+1.4r_n}dx_2\int_{0.5+0.8r_n}^{0.5+0.9r_n}dx_3\\ \label{varexpress3}
&= (0.002)c^3\frac{\mu_n^2}{9} r_n^3.
\end{align}
Recall that $\mu_n=\frac{3}{4}+o(1)$. Then $\mathbb{E}\left[h_2^2(X_1,X_2)\right]\geq Cr_n^3$ for a positive constant $C$ and large $n$.

Next, we find an upper bound  for $\mathbb{E}\left[h_2^2(X_1,X_2)\right]$. By the properties of conditional expectation, we have
\begin{eqnarray}\nonumber
\mathbb{E}\left[h_2^2(X_1,X_2)\right]&=&\mathbb{E}\big[\mathbb{E}[h(X_1,X_2,X_3)|X_1,X_2]\mathbb{E}[h(X_1,X_2,X_4)|X_1,X_2]\big]\\ \nonumber
&=&\mathbb{E}\big[\mathbb{E}[h(X_1,X_2,X_3)h(X_1,X_2,X_4)|X_1,X_2]\big]\\ \label{varexpress4}
&=&\mathbb{E}\big[h(X_1,X_2,X_3)h(X_1,X_2,X_4)\big].
\end{eqnarray}
From the definition of $h(X_1,X_2,X_3)$ in (\ref{defhf}), straightforward calculation yields
\begin{eqnarray}\nonumber
&&h(X_1,X_2,X_3)h(X_1,X_2,X_4)\\ \nonumber
&\leq&A_{12}A_{13}A_{23}A_{14}A_{24}\\ \nonumber
&&+\frac{\mu_n}{3}A_{12}A_{13}A_{23}(A_{12}A_{24}+A_{12}A_{14}+A_{14}A_{42})\\ \nonumber
&&+\frac{\mu_n}{3}A_{12}A_{14}A_{24}(A_{12}A_{23}+A_{12}A_{13}+A_{13}A_{32})\\  \label{ubefq1}
&&+\frac{\mu_n^2}{9}(A_{12}A_{23}+A_{12}A_{13}+A_{13}A_{32})(A_{12}A_{24}+A_{12}A_{14}+A_{14}A_{42}).
\end{eqnarray}
Note that $0\leq A_{ij}\leq1$.
By Proposition \ref{propmain}, we have
\begin{align}\label{ubefq2}
    \mathbb{E}[A_{12}A_{13}A_{23}A_{14}A_{24}]\leq \mathbb{E}[A_{12}A_{13}A_{14}]=\mathbb{E}[(\mathbb{E}[A_{12}|X_1])^3]=O(r_n^3).
\end{align}
According to  Proposition \ref{propmain}, for large $n$, one has $\mu_n\leq 1$. Following a similar argument to that in   (\ref{ubefq2}), it is easy to show that the expectations of the remaining terms in (\ref{ubefq1}) are of order $O(r_n^3)$. In view of (\ref{varexpress4}), we get
\begin{align}\label{ubefq3}
\mathbb{E}\left[h_2^2(X_1,X_2)\right]=\mathbb{E}[h(X_1,X_2,X_3)h(X_1,X_2,X_4)]=O(r_n^3).
\end{align}
Combining (\ref{ubefq3}) and (\ref{varexpress3}) yields   (\ref{m1edeq2})   of Lemma \ref{varlem2}.

We will use Lemma \ref{ustclt} to derive the asymptotic distribution of $\sum_{i\neq j}h_2(X_i,X_j)$. Suppose $nr_n^5=o(1)$. Firstly, we verify condition (\ref{ustclt1}). According to (\ref{ubefq3}) and Lemma \ref{varlem}, $v_n^2=\Theta(n^2r_n^3+n^3r_n^8)$. By Proposition \ref{propmain} and the definition of $h_2(X_1, X_2)$, we have
\begin{align}\nonumber
    |h_2(X_1, X_2)|&=|\mathbb{E}[h(X_1, X_2, X_3)|X_1,X_2]| \\ \nonumber
    &\leq \mathbb{E}[A_{12}A_{23}A_{13}|X_1,X_2] + \frac{\mu_n}{3}\big(\mathbb{E}[A_{12}A_{23}|X_1,X_2] + \mathbb{E}[A_{12}A_{13}|X_1,X_2] \\ \nonumber
    &\quad+ \mathbb{E}[A_{23}A_{13}|X_1,X_2]\big),\\ \nonumber
    &\leq \mathbb{E}[A_{13}|X_1,X_2]+\frac{\mu_n}{3}\big(\mathbb{E}[A_{23}|X_1,X_2] + \mathbb{E}[A_{13}|X_1,X_2] \\ \nonumber
    &\quad+ \mathbb{E}[A_{13}|X_1,X_2]\big),\\ \label{defh22f}
    &=O(r_n),
\end{align}
where $O(r_n)$ does not depend on $X_1,X_2$. Then
\begin{align*}
    \frac{\sup_{x,y}|h_2(x, y)|}{v_n}=O\left(\frac{r_n}{\sqrt{n^2r_n^3+n^3r_n^8}}\right)=O\left(\frac{\sqrt{r_n}}{nr_n\sqrt{1+nr_n^5}}\right)=o(1).
\end{align*}
Then condition (\ref{ustclt1}) holds.

Next, we verify condition (\ref{ustclt2}).
Note that
\begin{align}\nonumber
    |h_2(X_1, X_2)|
    &\leq \mathbb{E}[A_{23}A_{13}|X_1,X_2] + \frac{\mu_n}{3}\big(\mathbb{E}[A_{12}A_{23}|X_1,X_2] + \mathbb{E}[A_{12}A_{13}|X_1,X_2] \\ \nonumber
    &\quad+ \mathbb{E}[A_{23}A_{13}|X_1,X_2]\big).
\end{align}
Then, by (\ref{defh22f}) and Proposition \ref{propmain}, we have
 \begin{align*}\nonumber
    &\mathbb{E}[|h_2(X_1, X_2)|\big|X_1]\\ \nonumber
    &\leq \mathbb{E}[\mathbb{E}[A_{23}A_{13}|X_1,X_2]\big|X_1] + \frac{\mu_n}{3}\big(\mathbb{E}[\mathbb{E}[A_{12}A_{23}|X_1,X_2]\big|X_1] + \mathbb{E}[\mathbb{E}[A_{12}A_{13}|X_1,X_2]\big|X_1] \\ \nonumber
    &\quad+ \mathbb{E}[\mathbb{E}[A_{23}A_{13}|X_1,X_2]\big|X_1]\big)\\ \nonumber
    &=\mathbb{E}[A_{23}A_{13}|X_1] + \frac{\mu_n}{3}\big(\mathbb{E}[A_{12}A_{23}|X_1] + \mathbb{E}[A_{12}A_{13}|X_1] \\ \nonumber
    &\quad+\mathbb{E}[A_{23}A_{13}|X_1]\big)\\
    &=O(r_n^2),
\end{align*}
where $O(r_n^2)$ does not depend on $X_1$.
It then follows that
\begin{align*}
    \frac{n\sup_{x}\mathbb{E}[|h_2(x, X_1)|]}{v_n}=O\left(\frac{nr_n^2}{\sqrt{n^2r_n^3+n^3r_n^8}}\right)=O\left(\frac{\sqrt{r_n}}{\sqrt{1+nr_n^5}}\right)=o(1).
\end{align*}
Then condition (\ref{ustclt2}) holds. Then (\ref{m1edeq1}) follows from Lemma \ref{ustclt}.

When $f(x)$ is the uniform density, it is easy to verify that $h_1(X_1)=0$. In this case, $v_n^2=\frac{n^2}{2}\mathbb{E}[h_2^2(X_1,X_2)]=\Theta(n^2r_n^3)$. The proof for the case $nr_n^5=o(1)$ still works. Then the proof of  Lemma \ref{varlem2} is complete.

\qed

\begin{Lemma}\label{varlem3}
    Under the assumption of Proposition \ref{propmain}, we have
   \begin{align}\label{var1lem3}
       \sigma_{3n}^2=\mathbb{E}[h^2(X_1,X_2,X_3)]= \frac{3r_n^2}{8}\mathbb{E}[f^2(X_1)]+O(r_n^4).
   \end{align} 
   If $nr_n=o(1)$ and $n^3r_n^2=\omega(1)$, then
    \begin{align}\label{var2lem3}
    \frac{\sum_{i\neq j\neq k}h(X_1,X_2,X_3)}{\sqrt{6}n\sqrt{n}\sigma_{3n}}\Rightarrow N(0,1).
   \end{align} 
   where $\sigma_{3n}^2=\mathbb{E}[h^2(X_1,X_2,X_3)]$.
\end{Lemma}

{\bf Proof of Lemma \ref{varlem3}:} Firstly, we prove (\ref{var1lem3}). By the definition of $h(X_1,X_2,X_3)$ in (\ref{defhf}), we have
\begin{align*} 
    h^2(X_1, X_2, X_3) &= \Delta_{123} -2\mu_n\Delta_{123}+\frac{\mu_n^2}{9}(P_{123} + P_{213} + P_{231})+\frac{2\mu_n^2}{3}\Delta_{123}.
\end{align*}
In view of Proposition \ref{propmain},  taking the expectation of both sides of the previous equation yields
\begin{align*}
    \mathbb{E}[h^2(X_1,X_2,X_3)]&=\left(1-2\mu_n+\frac{2\mu_n^2}{3}\right) \mathbb{E}[\Delta_{123}]+ \frac{\mu_n^2}{3} \mathbb{E}[P_{123}]\\
    &=\left(1-\frac{3}{2}+\frac{3}{8}\right)3r_n^2\mathbb{E}[f^2(X_1)]+\frac{3}{4}r_n^2\mathbb{E}[f^2(X_1)]+O(r_n^4)\\
    &=\frac{3r_n^2}{8}\mathbb{E}[f^2(X_1)]+O(r_n^4).
\end{align*}
Then (\ref{var1lem3}) holds.

 We will use the method of moments to prove (\ref{var2lem3}). Specifically, we show that the moments of $ \frac{\sum_{i\neq j\neq k}h(X_1,X_2,X_3)}{n\sqrt{6n}\sigma_{3n}}$ converges to the moments of the standard normal distribution.  For convenience, let $J_s=(i_s,j_s,k_s)$, where $i_s,j_s,k_s\in\{1,2,3,\dots,n\}$ are pairwise distinct. Denote $h_{J_s}=h(X_{i_s},X_{j_s},X_{k_s})$.  Let $m$ be a positive integer.

Firstly, we show that the even-order moments converge to those of the standard normal distribution. 
The $2m$-th moments of $ \frac{\sum_{i\neq j\neq k}h(X_1,X_2,X_3)}{n\sqrt{6n}\sigma_{3n}}$ can be expressed as
\begin{align}\label{mmenc}
    \mathbb{E}\left[\left(\frac{\sum_{i\neq j\neq k}h(X_1,X_2,X_3)}{n\sqrt{6n}\sigma_{3n}}\right)^{2m}\right]&=\frac{\sum_{J_1,J_2,\dots,J_{2m}}\mathbb{E}[h_{J_1}h_{J_2}\dots h_{J_{2m}}]}{(6n^3\sigma_{3n}^2)^m}.
\end{align}
Given $s\in\{1,2,3\dots, 2m\}$, if $J_s \cap J_l =\emptyset$ for all $l \in \{1,2,\dots, 2m\}\setminus\{s\}$, then $h_{J_s}$ is independent of $h_{J_l}$ for all $l \in \{1,2,\dots, 2m\}\setminus\{s\}$. Recall that $\mathbb{E}[h_{J_s}]=0$. In this case,
\[\mathbb{E}\left[h_{J_1} h_{J_2} \cdots h_{J_{2m}}\right] = \mathbb{E}[h_{J_s}] \mathbb{E}\left[\prod_{l \in\{1,2,\dots, 2m\}\setminus\{s\}} h_{J_l}\right] = 0.\]
Therefore, $J_s \cap J_{l_0} \neq \emptyset$ for some $l_0 \in \{1,2,\dots, 2m\}\setminus\{s\}$. The expectation in (\ref{mmenc}) is non-zero if and only if every $J_{s}$  has a non-empty intersection with some $J_{t}$ ($s\neq t$). Then, the collection $\{J_1, J_2, \ldots, J_{2m}\}$  can be partitioned into $t$ ($1\leq t\leq m$) disjoint components, where each component consists of triples that share at least one index with at least one other triple in the same component. For each $l\in\{1,2,\dots,t\}$,  let $C_l = \{J_{s_{l_1}}, J_{s_{l_2}}, \ldots, J_{s_{l_{m_l}}}\}$ be the $l$-th connected component, where $s_{lq} \in \{1,2,3\dots, 2m\}$, $1\leq q\leq m_l$ and $2\leq m_l\leq 2m$.

Suppose $t=m$. In this case, each connected component $C_l$ ($1\leq l\leq t$) contains exactly two identical triples. That is,  the collection $\{J_1, J_2, \ldots, J_{2m}\}$ 
 are partitioned into $m$
 disjoint pairs $\{J_s,J_t\}$
 such that $J_s=J_t$. There are $(2m-1)!!$ such partitions. The expectation in (\ref{mmenc}) is identical for each such partition. Without loss of generality, let $J_t=J_{m+t}$ for $1\leq t\leq m$ and $J_{t_1}\cap J_{t_2}=\emptyset$ for distinct $t_1,t_2\in\{1,2,\dots,m\}$. There are 6 possible ways for the set of indices $J_t=\{i_t,j_t,k_t\}$ to equal the set of indices $J_{m+t}=\{i_{m+t},j_{m+t},k_{m+t}\}$. Moreover, $h_{J_1},h_{J_2}, \dots, h_{J_m}$  are independent.  Recall that $\sigma_{3n}^2=\mathbb{E}[h_{J_1}^2]$. Then
\begin{align}\nonumber
    \frac{\sum_{J_1,J_2,\dots,J_{2m}}\mathbb{E}[h_{J_1}h_{J_2}\dots h_{J_{2m}}]}{(6n^3\sigma_{3n}^2)^m}&=(2m-1)!!\frac{6^m\sum_{J_1,J_2,\dots,J_{m}}\mathbb{E}[h_{J_1}^2h_{J_2}^2\dots h_{J_{m}}^2]}{(6n^3\sigma_{3n}^2)^m}\\ \nonumber
    &=(2m-1)!!\frac{6^m\prod_{t=1}^{3m}(n-t+1)}{(6n^3\sigma_{3n}^2)^m}\mathbb{E}[h_{J_1}^2]\mathbb{E}[h_{J_2}^2]\dots \mathbb{E}[h_{J_{m}}^2]\\ \nonumber
    &=(2m-1)!!\frac{n^{3m}+O(n^{3m-1})}{n^{3m}}\\ \label{mmenc1}
    &=(2m-1)!!+o(1).
\end{align}

Suppose $1\leq t\leq m-1$. Since the connected components $C_l$ ($1\leq l\leq t$)  are disjoint, then $\prod_{I_s \in C_1} h_{I_s}$, \dots, $\prod_{I_s \in C_t} h_{I_s}$ are independent. Then
\[\mathbb{E}[h_{J_1}h_{J_2}\ldots h_{J_{2m}}] = \mathbb{E}\left[\prod_{J_s \in C_1} h_{J_s}\right] \dots \mathbb{E}\left[\prod_{J_s \in C_t} h_{J_s}\right].\]
By the definition of $\mu_n$, we have $\frac{\mu_n}{3}\leq1$ for all $n$.
Then the absolute value of each product term  in the preceding equation is bounded by
\begin{align*}
\left|\mathbb{E}\left[\prod_{J_s \in C_l} h_{J_s}\right]\right| &\leq \mathbb{E}\left[\prod_{J_s \in C_l} |h_{J_s}|\right]\leq \mathbb{E}\left[\prod_{J_s \in C_l} (\Delta_{i_s j_s k_s} + P_{i_s j_s k_s} + P_{i_s k_s j_s} + P_{j_s i_s k_s})\right].
\end{align*}
The product $\prod_{J_s \in C_l} (\Delta_{i_s j_s k_s} + P_{i_s j_s k_s} + P_{i_s k_s j_s} + P_{j_s i_s k_s})$ expands into a sum of $4^{m_l}$ terms. Each such term is a product of the form $H_1H_2\dots H_{m_l}$, where $H_s\in\{\Delta_{i_s j_s k_s}, P_{i_s j_s k_s}, P_{j_s i_s k_s}, P_{i_s k_s j_s}\}$ for $1\leq s\leq m_l$. Recall that the triples in $C_l$ are connected; that is, any two triples in $C_l$ share at least one common vertex. Suppose $C_l$ has $v_l$ distinct vertices. There are at most $n^{v_l}$ choices of such vertices. When $H_1H_2\dots H_{m_l}=1$, the union $\bigcup_{l=1}^{m_l} H_l$ forms a connected graph, denoted by $G_l$. Then $G_l$ has $v_l$ vertices. There exists a spanning tree $T_l$ for $G_l$ with exactly $v_l-1$ edges; let $\mathcal{E}(T_l)$ denote its edge set. Since $0\leq A_{ij}\leq 1$, then
\[\mathbb{E}[H_1H_2\dots H_{m_l}]\leq \mathbb{E}\left[\prod_{e\in \mathcal{E}(T_l)}A_e\right]=O(r_n^{v_l-1}),\]
where the last equality is obtained by repeatedly applying (\ref{prope0}) of Proposition \ref{propmain} $v_l-1$ times. Since $m$ is a fixed positive integer, then $4^{m_l}$ and $t$ are fixed constants. Then
\[\left|\mathbb{E}\left[\prod_{J_s \in C_1} h_{J_s}\right] \dots \mathbb{E}\left[\prod_{J_s \in C_t} h_{J_s}\right]\right|=O(r_n^{v_1+v_2+\dots+v_t-t}).\]
Therefore, we have
\begin{align}\label{mmenc2}
    \frac{\sum_{C_1,\dots,C_t}\mathbb{E}\left[\prod_{J_s \in C_1} h_{J_s}\right] \dots \mathbb{E}\left[\prod_{J_s \in C_t} h_{J_s}\right]}{(6n^3\sigma_{3n}^2)^m}&=O\left(\frac{n^{v_1+v_2+\dots+v_t}r_n^{v_1+v_2+\dots+v_t-t}}{(6n^3\sigma_{3n}^2)^m}\right).
\end{align}
Since the triples in  $C_l$ are connected and the three indices within each triple are distinct, it follows that
$3 \leq v_l \leq 3m_l-1$. Then $3t \leq v_1 + \dots + v_t \leq 3(m_1 + \ldots + m_t) - t$. Since $m_1 + m_2 + \ldots + m_t = 2m$, then $3t \leq v_1 + \dots + v_t \leq 6m - t$. Let $w=v_1 + \dots + v_t$. Given that  $nr_n=o(1)$, then
\[\frac{n^{w}r_n^{w-t}}{n^{w+1}r_n^{w+1-t}}=\frac{1}{nr_n}=\omega(1),\]
which implies $n^{w+1}r_n^{w+1-t}=o(n^{w}r_n^{w-t})$. Recalling from (\ref{var1lem3}) that   $\sigma_{3n}^2=\Theta(r_n^2)$, and noting that  $1\leq t\leq m-1$ with $n^3r_n^2=\omega(1)$ by assumption, we obtain
\begin{align}\label{mmenc3}
\frac{n^{v_1+v_2+\dots+v_t}r_n^{v_1+v_2+\dots+v_t-t}}{(n^3\sigma_{3n}^2)^m}=O\left(\frac{n^{3t} r_n^{3t-t}}{n^{3m} r^{2m}_n} \right)= O\left(\frac{(n^3 r_n^2)^t}{(n^3r^2_n)^m}\right) =O\left(\frac{1}{n^3r^2_n}\right)= o(1).
\end{align}

Combining (\ref{mmenc})-(\ref{mmenc3}) yields
\begin{align}\label{mm11enc}
    \mathbb{E}\left[\left(\frac{\sum_{i\neq j\neq k}h(X_1,X_2,X_3)}{n\sqrt{6n}\sigma_{3n}}\right)^{2m}\right]&=(2m-1)!!+o(1).
\end{align}
The $2m$-th moment of $ \frac{\sum_{i\neq j\neq k}h(X_1,X_2,X_3)}{n\sqrt{6n}\sigma_{3n}}$  converges to the $2m$-th moment of the standard normal distribution.

Next, we show the odd-order moments of$ \frac{\sum_{i\neq j\neq k}h(X_1,X_2,X_3)}{n\sqrt{6n}\sigma_{3n}}$ converges to zero. The $(2m+1)$-th moments of $ \frac{\sum_{i\neq j\neq k}h(X_1,X_2,X_3)}{n\sqrt{6n}\sigma_{3n}}$ can be expressed as
\begin{align*}
    \mathbb{E}\left[\left(\frac{\sum_{i\neq j\neq k}h(X_1,X_2,X_3)}{n\sqrt{6n}\sigma_{3n}}\right)^{2m+1}\right]&=\frac{\sum_{J_1,J_2,\dots,J_{2m+1}}\mathbb{E}[h_{J_1}h_{J_2}\dots h_{J_{2m+1}}]}{n\sqrt{6n}\sigma_{3n}(n^3\sigma_{3n}^2)^m}.
\end{align*}
Let $C_l$ ($1\leq l\leq t$) be defined as in the even-order case. In this case, $1\leq t\leq m$. By a similar argument as in (\ref{mmenc2}) and (\ref{mmenc3}), we have
\begin{align*}
    \mathbb{E}\left[\left(\frac{\sum_{i\neq j\neq k}h(X_1,X_2,X_3)}{n\sqrt{6n}\sigma_{3n}}\right)^{2m+1}\right]&= O\left(\frac{(n^3 r_n^2)^t}{n\sqrt{n}r_n(n^3r^2_n)^m}\right) =O\left(\frac{1}{\sqrt{n^3r_n^2} }\right)= o(1).
\end{align*}
Then $(2m+1)$-th moments of $ \frac{\sum_{i\neq j\neq k}h(X_1,X_2,X_3)}{n\sqrt{6n}\sigma_{3n}}$ converges to zero. In view of (\ref{mm11enc}), it follows that the moments of the normalized sum converge to those of the standard normal distribution. Consequently, we conclude that (\ref{var2lem3}) holds, completing the proof.

\qed

\subsection{Proof of Theorem \ref{mthm}}

Recall that $\Delta_{123}=A_{12}A_{23}A_{31}$, $P_{123}=A_{12}A_{23}$, and $\mu_n=\frac{\mathbb{E}[\Delta_{123}]}{\mathbb{E}[P_{123}]}$. Define a U-statistic $U_n$ as follows
\begin{align*}
    U_n = \sum_{i \neq j \neq k} h(X_i, X_j, X_k),
\end{align*}
where the function $h(X_1, X_2, X_3)$ is given in (\ref{defhf}). Then the global clustering coefficient $\mathcal{C}_n$ can be expressed in terms of $U_n$ as follows:
\begin{align}\label{cnasun}
    \left(\sum_{i\neq j\neq k}A_{ij}A_{jk}\right)(\mathcal{C}_n-\mu_n)=U_n.
\end{align}
We will determine the leading-order term of $U_n$ and derive its asymptotic distribution. To this end, we evaluate its variance.

It is straightforward to verify that $\mathbb{E}[h(X_1, X_2, X_3)]=0$. Moreover, if the sets of indices $\{i,j,k\}$ and $\{l,s,t\}$ are disjoint, then  $h(X_i, X_j, X_k)$ and  $h(X_l, X_s, X_t)$ are independent. In this case, we have
\[\mathbb{E}[h(X_i, X_j, X_k) h(X_l, X_s, X_t)]=\mathbb{E}[h(X_i, X_j, X_k)]\mathbb{E}[h(X_l, X_s, X_t)]=0.\]
Then the variance of $U_n$ can be written as
\begin{align*}
\text{Var}(U_n) &= \sum_{\substack{i \neq j \neq k \\ l \neq s \neq t}} \mathbb{E}[h(X_i, X_j, X_k) h(X_l, X_s, X_t)]\\
&= 9 \sum_{\substack{i \neq j \neq k\neq s\neq l}} \mathbb{E}[h(X_i, X_j, X_k) h(X_i, X_l, X_s)] + 18 \sum_{\substack{l\neq i \neq j \neq k }} \mathbb{E}[h(X_i, X_j, X_k) h(X_i, X_j, X_{l})]\\
&\quad + 6 \sum_{i \neq j \neq k} \mathbb{E}[h^2(X_i, X_j, X_k)].
\end{align*}
Note that
\begin{align*}
\mathbb{E}[h(X_i, X_j, X_k) h(X_i, X_l, X_s)]&=\mathbb{E}\left[\mathbb{E}[h(X_i, X_j, X_k)|X_i] \mathbb{E}[h(X_i, X_l, X_{s})|X_i]\right]=\mathbb{E}[h_1^2(X_i)],\\
\mathbb{E}[h(X_i, X_j, X_k) h(X_i, X_j, X_{l})]&=\mathbb{E}\left[\mathbb{E}[h(X_i, X_j, X_k)|X_i,X_j] \mathbb{E}[h(X_i, X_j, X_{l})|X_i,X_j]\right]\\
&=\mathbb{E}[h_2^2(X_i, X_j)].
\end{align*}
Then the variance of $U_n$ takes the form:
\begin{align}\nonumber
\text{Var}(U_n) &=9n^5\mathbb{E}[h_1^2(X_1)] + 18n^4\mathbb{E}[h_2^2(X_1, X_2)]+ 6n^3\mathbb{E}[h^2(X_1, X_2, X_3)]\\ \label{pfunvareq1}
&\quad + O(n^4)\mathbb{E}[h_1^2(X_1)] + O(n^3)\mathbb{E}[h_2^2(X_1, X_2)]+O(n^2)\mathbb{E}[h^2(X_1, X_2, X_3)].
\end{align}

{\bf Proof of Case (I)}. Suppose $nr_n^5 \to \infty$, then $n^4r_n^3=o(n^5 r_n^8)$ and $n^3 r_n^2=o(n^5 r_n^8)$.  By Lemma \ref{varlem}, Lemma \ref{varlem2} and Lemma \ref{varlem3}, the first term of (\ref{pfunvareq1}) is the leading term. We show 
 $U_n$ is asymptotically equal to $3n^2 \sum_{i} h_1(X_i)$. Note that
\begin{align}\nonumber
&\mathbb{E}\left[\left(U_n - 3n^2 \sum_{i} h_1(X_i)\right)^2\right]\\ \label{expuh0}
&= \mathbb{E}[U_n^2] - 6n^2 \mathbb{E}\left[U_n \sum_{i} h_1(X_i)\right] + 9n^4 \mathbb{E}\left[\left(\sum_{i} h_1(X_i)\right)^2\right].
\end{align}
Next, we simplify the expectations in (\ref{expuh0}) to show that their sum is of order $O(n^4 r_n^{8})$.
By (\ref{pfunvareq1}), one has
\begin{align}\label{expuh00}
    \mathbb{E}[U_n^2]=9n^5 \mathbb{E}[h_1^2(X_1)]+O(n^4 r_n^3 + n^3 r_n^2 + n^2 r_n^2).
\end{align}
The second expectation of (\ref{expuh0}) is equal to
\begin{align}\label{expuh1}
\mathbb{E}\left[U_n \sum_{i} h_1(X_i)\right]
&= \sum_{j\neq k\neq l, i}\mathbb{E}\left[h(X_j,X_k,X_l)h_1(X_i)\right].
\end{align}
If $i\not\in\{j,k,l\}$, then $X_i$ is independent of $X_j,X_k,X_l$. In this case,
\[\mathbb{E}\left[h(X_j,X_k,X_l)h_1(X_i)\right]=\mathbb{E}\left[h(X_j,X_k,X_l)\right]\mathbb{E}\left[h_1(X_i)\right]=0.\]
Therefore, we get $i\in\{j,k,l\}$. The expectations in (\ref{expuh1}) are identical for all three cases $i=j$, or $i=k$, or $i=l$. Without loss of generality, let $i=j$. Then
\begin{align*}\nonumber
\mathbb{E}\left[h(X_j,X_k,X_l)h_1(X_i)\right]&= \mathbb{E}\left[\mathbb{E}\left[h(X_j,X_k,X_l)h_1(X_j)|X_j\right]\right]\\ \nonumber
&=\mathbb{E}\left[h_1(X_j)\mathbb{E}\left[h(X_j,X_k,X_l)|X_j\right]\right]\\ 
&=\mathbb{E}\left[h_1^2(X_j)\right].
\end{align*}
It then follows from (\ref{expuh1}) that
\begin{align}\label{expuh3}
\mathbb{E}\left[U_n \sum_{i} h_1(X_i)\right]
&= 3n^3\mathbb{E}\left[h_1^2(X_j)\right]+O(n^2)\mathbb{E}\left[h_1^2(X_j)\right].
\end{align}

Note that $\mathbb{E}\left[h_1(X_i)h_1(X_j)\right]=\mathbb{E}\left[h_1(X_i)\right]\mathbb{E}\left[h_1(X_j)\right]=0$ if $i\neq j$. Then the last expectation of (\ref{expuh0}) is equal to
\begin{align}\label{eoxpuh3}
    \mathbb{E}\left[\left(\sum_{i} h_1(X_i)\right)^2\right]&=\sum_{i,j} \mathbb{E}\left[h_1(X_i)h_1(X_j)\right]=\sum_{i} \mathbb{E}\left[h_1^2(X_i)\right]=n\mathbb{E}\left[h_1^2(X_1)\right].
\end{align}

In view of (\ref{expuh0}), (\ref{expuh00}), (\ref{expuh3}) and (\ref{eoxpuh3}), we have
\begin{align*} 
\mathbb{E}\left[\left(U_n - 3n^2 \sum_{i} h_1(X_i)\right)^2\right]= O(n^4 r_n^{3}).
\end{align*}
Then we conclude that
$U_n =3n^2 \sum_{i} h_1(X_i)+O_P(n^2 r_n\sqrt{r_n})$.
By Lemma \ref{varlem}, one has
\begin{align*}
    \frac{4U_n}{3n^2\sqrt{n}r_n^4\sigma_1}=\frac{ 4\sum_{i} h_1(X_i)}{\sqrt{n}r_n^4\sigma_1}+O_P\left(\frac{1}{\sqrt{nr_n^5}}\right)\Rightarrow N(0,1).
\end{align*}
In view of (\ref{cnasun}) and Lemma \ref{00varlem},  the result in (I) follows.

{\bf Proof of Case (II)}. Suppose $nr_n^5=o(1)$, $nr_n=\omega(1)$ and $f(x)$ \ is not the uniform density. Then $n^3r_n^2=o(n^4 r_n^3 )$ and $n^5r_n^8=o(n^4 r_n^3)$. By Lemma \ref{varlem}, Lemma \ref{varlem2} and Lemma \ref{varlem3}, the second term of (\ref{pfunvareq1}) is the leading term. We will show that $U_n$ is asymptotically equal to $3n \sum_{i \neq j} h_2(X_i,X_j)$. It is easy to verify that
\begin{align}\nonumber
&\mathbb{E}\left[\left(U_n - 3n \sum_{i \neq j} h_2(X_i,X_j)\right)^2\right]\\ \label{sunneq1}
&= \mathbb{E}[U_n^2] - 6n \mathbb{E}\left[U_n \sum_{i \neq j} h_2(X_i, X_j)\right] + 9n^2 \mathbb{E}\left[\left(\sum_{i \neq j} h_2(X_i, X_j)\right)^2\right].
\end{align}
By (\ref{pfunvareq1}), the first expectation of (\ref{sunneq1}) is equal to
\begin{align}\label{sunneq2}
 \mathbb{E}[U_n^2] =18n^4\mathbb{E}[h_2^2(X_1, X_2)]+ O(n^5r_n^8+n^3r_n^2).
\end{align}

The second expectation of (\ref{sunneq1}) is expressed as
\begin{align*}
\mathbb{E}\left[U_n \sum_{i \neq j} h_2(X_i, X_j)\right]=\sum_{i_1\neq j_1\neq k_1, i \neq j} \mathbb{E}\left[h(X_{i_1},X_{j_1},X_{k_1})h_2(X_i, X_j)\right].
\end{align*}
If $\{i_1,j_1, k_1\}\cap\{i,j\}=\emptyset$, then 
\[ \mathbb{E}\left[h(X_{i_1},X_{j_1},X_{k_1})h_2(X_i, X_j)\right]= \mathbb{E}\left[h(X_{i_1},X_{j_1},X_{k_1})\right] \mathbb{E}\left[h_2(X_i, X_j)\right]=0.\]
Suppose $|\{i,j\}\cap\{i_1,j_1,k_1\}|=1$. There are at most $n^4$ such indices. Without loss of generality, let $i=i_1$. Then
\[ \mathbb{E}\left[h(X_{i},X_{j_1},X_{k_1})h_2(X_i, X_j)\right]= \mathbb{E}\left[\mathbb{E}\left[h(X_{i},X_{j_1},X_{k_1})h_2(X_i, X_j)|X_i\right]\right]=\mathbb{E}\left[h_1^2(X_i)\right]=O(r_n^8).\]
Suppose $\{i,j\}\subset\{i_1,j_1,k_1\}$. There are at most $n^3$ such indices. Without loss of generality, let $\{i,j\}=\{i_1,j_1\}$.
\begin{align*} 
\mathbb{E}\left[h(X_{i_1},X_{j_1},X_{k_1})h_2(X_i, X_j)\right]&= \mathbb{E}\left[\mathbb{E}\left[h(X_{i},X_{j},X_{k_1})h_2(X_i, X_j)|X_i,X_j\right]\right]=\mathbb{E}\left[h_2^2(X_i, X_j)\right].
\end{align*}
This result holds for the other cases such as $\{i,j\}=\{i_1,k_1\}$. 
Therefore, we have
\begin{align}\label{sunneq3}
\mathbb{E}\left[U_n \sum_{i \neq j} h_2(X_i, X_j)\right]=6n^3\mathbb{E}\left[h_2^2(X_i, X_j)\right]+O(n^4r_n^8+n^2r_n^3).
\end{align}

The third expectation of (\ref{sunneq1}) is expressed as
\begin{align*} 
\mathbb{E}\left[\left(\sum_{i \neq j} h_2(X_i, X_j)\right)^2\right]&=\sum_{i_1\neq j_1, i \neq j}\mathbb{E}\left[h_2(X_i, X_j)h_2(X_{i_1}, X_{j_1})\right]
\end{align*}
If $\{i,j\}\cap\{i_1,j_1\}=\emptyset$, then $h_2(X_i, X_j)$ and $h_2(X_{i_1}, X_{j_1})$ are independent. In this case,
\[\mathbb{E}\left[h_2(X_i, X_j)h_2(X_{i_1}, X_{j_1})\right]=\mathbb{E}\left[h_2(X_{i_1}, X_{j_1})\right]\mathbb{E}\left[h_2(X_i, X_j)\right]=0\]
Suppose $\{i,j\}=\{i_1,j_1\}$. There are at most $n^2$ such indices, and
there are two ways for the set of indices $\{i,j\}$ to equal $\{i_1,j_1\}$. Then
\[\mathbb{E}\left[h_2(X_i, X_j)h_2(X_{i_1}, X_{j_1})\right]=\mathbb{E}\left[h_2^2(X_i, X_j)\right]=\Theta(r_n^3).\]

Suppose $|\{i,j\}\cap\{i_1,j_1\}|=1$. There are at most $n^3$ such indices. Without loss of generality, let $i=i_1$. In this case, we have
\begin{align*}
    \mathbb{E}\left[h_2(X_i, X_j)h_2(X_{i_1}, X_{j_1})\right]&=\mathbb{E}\left[\mathbb{E}\left[h_2(X_i, X_j)h_2(X_{i}, X_{j_1})|X_i\right]\right]=\mathbb{E}\left[h_1^2(X_1)\right]=O(r_n^8).
\end{align*}
Then
\begin{align} \label{sunneq4}
\mathbb{E}\left[\left(\sum_{i \neq j} h_2(X_i, X_j)\right)^2\right]&=2n^2\mathbb{E}\left[h_2^2(X_1, X_2)\right]+O(n^3r_n^8+n^2r_n^3).
\end{align}

Combining (\ref{sunneq1})-(\ref{sunneq4}) yields
\begin{align*}
\mathbb{E}\left[\left(U_n - 3n \sum_{i \neq j} h_2(X_i,X_j)\right)^2\right]=O(n^5 r_n^8 + n^3 r_n^2 ),
\end{align*}
from which it follows that
\[U_n = 3n \sum_{i \neq j} h_2(X_i,X_j) + O_P(\sqrt{n^5 r_n^8 + n^3 r_n^2}).\]
By Lemma \ref{varlem2}, we have
\begin{align*}
    \frac{\sqrt{2}U_n}{6n^2\sigma_{2n}}= \frac{\sqrt{2}\sum_{i \neq j} h_2(X_i,X_j)}{2n\sigma_{2n}}+O_P\left(\sqrt{nr_n^5+\frac{1}{nr_n}}\right)\Rightarrow N(0,1).
\end{align*}

In view of Lemma \ref{00varlem} and (\ref{cnasun}), we complete the proof for  the case $nr_n^5=o(1)$, $nr_n=\omega(1)$ and $f(x)$ \ is not the uniform density.  When $f(x)$ is the uniform density, $h_1(x)=0$. In this case, the previous proof still works.  Then proof of Case (II) is complete.

{\bf Proof of Case (III)}. The result of Case (III) follows from Lemma \ref{00varlem}, equation (\ref{cnasun}) and Lemma \ref{varlem3}. Then the proof of Theorem \ref{mthm} is complete.

\qed

\section*{Acknowledgement}
Mingao Yuan thanks The University of Texas at El Paso for providing generous startup funds.


\begin{thebibliography}{9}


\bibitem{A17} Abbe, E. (2017).
Community detection and stochastic block models: recent developments. 
\textit{Journal of Machine Learning Research}, \textbf{18}, 1-86.








\bibitem{BB24}
Bangachev,K. and Bresler, G.(2024).
Detection of $L_{\infty}$ geometry in random geometric graphs:suboptimality of triangles and cluster expansion. \textit{Proceedings of Machine Learning Research}, 247:1–71.







    
    \bibitem{BC23}Badiu, M.-A. and Coon, J. P.(2023),
    Structural complexity of one-dimensional random geometric graphs, \textit{IEEE Transactions on Information Theory},  69(2),794-812.
    
    \bibitem{Barthelemy2011} Barth\'elemy, M. (2011). Spatial networks. \textit{Physics Reports}, 499(1--3), 1--101.
 


    





\bibitem{DallChristensen2002} Dall, J. and Christensen, M. (2002). Random geometric graphs. \textit{Physical Review E}, 66(1), 016121.







\bibitem{DC23}
Duchemin, Q., De Castro, Y. (2023). Random geometric graph: some recent developments and perspectives. \textit{High Dimensional Probability IX. Progress in Probability} Birkhäuser, Cham.  

\bibitem{HRP08}
Desmond J. Higham, Marija Rašajski, Nataša Pržulj (2008). Fitting a geometric graph to a protein–protein interaction network, \textit{Bioinformatics}, 24(8), 1093–1099.


\bibitem{ErdosRenyi1959} Erd\H{o}s, P. and R\'enyi, A. (1959). On random graphs I. \textit{Publicationes Mathematicae}, 6, 290--297.





\bibitem{GRK05}
Goel, A., Rai, S., Krishnamachari, B.(2005). Monotone properties of random geometric graphs have sharp thresholds. \textit{Ann. Appl. Probab.} 15 (4) 2535-2552.

\bibitem{G21}
Ganesan, G., Robust paths in random geometric graphs with applications to mobile networks, \textit{2021 International Conference on COMmunication Systems \& NETworkS (COMSNETS)}, Bangalore, India, 2021, pp. 119-123.
    
 
\bibitem{Gilbert1961} Gilbert, E.N. (1961). Random plane networks. \textit{Journal of the Society for Industrial and Applied Mathematics}, 9(4), 533--543.


\bibitem{GMPS23}
Galhotra, S.,   Mazumdar, A., Pal, S., Saha, B.(2023).
Community recovery in the geometric block model, \textit{Journal of Machine Learning Research} 24: 1-53.



\bibitem{Gupta2000} Gupta, P. and Kumar, P.R. (2000). The capacity of wireless networks. \textit{IEEE Transactions on Information Theory}, 46(2), 388--404.

\bibitem{H08} Higham, D.J., Ra\v{s}ajski, M., and Pr\v{z}ulj, N. (2008). Fitting a geometric graph to a protein-protein interaction network. \textit{Bioinformatics}, 24(8), 1093--1099.




\bibitem{HM09}
Han, G. and Makowski, A.(2009). One-dimensional geometric random graphs with
nonvanishing densities—Part I: A strong zero-one
law for connectivity. \textit{IEEE Transactions on information theory}, 55(12),5832-5839.

\bibitem{HM12}
Han, G. and Makowski, A.(2012). One-dimensional geometric random graphs with
nonvanishing densities—Part II: a very strong zero-one
law for connectivity. \textit{Queueing Syst}, 72, 103-138.







\bibitem{JJ86}
Jammalamadaka, S.R. and Janson, S. (1986).
Limit theorems for a triangular scheme of U-statistics with applications to inter-point
distances, \textit{The Annals of Probability}, 14(4),1347-1358.


\bibitem{LCYC14}
Lee, D., etc. (2014), Analysis of clustering coefficients of online social
networks by duplication models, \textit{2014 IEEE International Conference on Communications (ICC)
}, 4095-4100.


\bibitem{N09}
Newman, M.E.J. (2009). Random Graphs with Clustering, \textit{Physical Review Letters}, 103, 058701.




    \bibitem {N03} Newman, M E J (2003),
    The structure and function of complex networks. \textit{SIAM Review} 45, 167-256 .




\bibitem{BM08}
 O'Malley, A. J. and Marsden,P. V. (2008). The analysis of social networks, \textit{Health Serv Outcomes Res Methodol}, {\bf 8}, 222-269.








 \bibitem{PBGKL23}
Paolino, R., Bojchevski, A., Gunnemann, S., Kutyniok, G. and Levie, R.(2023).
Unveiling the sampling density
in non-uniform geometric graphs, \textit{ICLR 2023}

    

\bibitem{Penrose2003} Penrose, M. (2003). \textit{Random Geometric Graphs}. Oxford University Press, Oxford.


\bibitem{RPW05}
Robins, G., P. Pattison, and J. Woolcock (2005). Small and other worlds: Global network
structures from local processes. \textit{American Journal of Sociology} 110 (4), 894–936.





\bibitem{SBL13}
Simpson, S., Bowman F. and Laurienti, P.(2013). Analyzing complex functional brain networks: Fusing statistics and network science to understand the brain,\textit{Statistics Surveys}, 7: 1-36.



\bibitem{WS98}
Watts,D. and Strogatz, S. (1998).
Collective dynamics of ‘small-world’ networks, \textit{Nature}, 393, 440-442. 






\bibitem{Y24}
Yuan, M. (2024). Asymptotic distribution of the friendship paradox of a random geometric graph. \textit{Brazilian Journal of Probability and Statistics},  38, 444-462.


    \bibitem{YF25}
    Yuan, M. and Yu, F. (2025). Hypothesis testing for the dimension of
random geometric graph, preprint, ResearchGate, DOI: 10.13140/RG.2.2.31959.53920

\bibitem{Y25c}
Yuan, M. (2025). Asymptotic distribution of the global
clustering coefficient in a random annulus
graph, https://arxiv.org/pdf/2510.15003


\bibitem{Y25}
Yuan, M. (2025b). Hypothesis testing for the uniformity of random geometric graph, https://arxiv.org/pdf/2510.14210

\bibitem{Y23b}
Yuan, M,(2025).  Limiting distribution for the Randic Index
of a random geometric graph, \textit{MATCH Commun. Math. Comput. Chem.},  93,767-789.


\bibitem{Y26}
Yuan, M. (2026). The weak law of large numbers for the friendship paradox index. https://arxiv.org/pdf/2602.10055




\bibitem{Zheng2025} Zheng, T., Zheng, X., Xue, B., Xiao, S., and Zhang, C. (2025). A network analysis of depressive symptoms and cognitive performance in older adults with multimorbidity: A nationwide population-based study. \textit{Journal of Affective Disorders}, 383, 78--86.



 

\bibitem{Z15}
Zhao,J. (2015). The absence of isolated node in geometric random graphs, \textit{53rd Annual Allerton Conference on Communication, Control, and Computing (Allerton)}, Monticello, IL, USA, 2015, pp. 881-886.


    
 
 


    


\bibitem{ZLKG14}
Zhang, W., Lim, C., Korniss, G. et al. Opinion dynamics and influencing on random geometric graphs. \textit{Sci Rep} 4, 5568 (2014).




\end{thebibliography}
\end{document}